\pdfoutput=1
\documentclass[11pt]{amsart}
\usepackage{amsmath,amssymb,amsthm,mathtools}
\usepackage{booktabs}
\usepackage{xcolor}
\usepackage[margin=1.18in]{geometry}
\usepackage[
  colorlinks=true,
  linkcolor=blue!60!black,
  citecolor=blue!60!black,
  urlcolor=blue!60!black,
  pdftitle={Exchange identities and symmetric slices of the valley Delta conjecture},
  pdfauthor={Henry Shin},
  pdfsubject={Algebraic combinatorics; Delta conjecture},
  pdfkeywords={Delta conjecture, valley version, diagonal harmonics, Macdonald polynomials, parking functions, q-binomial identities, transfer operators}
]{hyperref}
\usepackage[expansion=false]{microtype}
\usepackage{needspace}
\newtheorem{theorem}{Theorem}[section]
\newtheorem{lemma}[theorem]{Lemma}
\newtheorem{proposition}[theorem]{Proposition}
\newtheorem{corollary}[theorem]{Corollary}
\newtheorem{conjecture}[theorem]{Conjecture}
\theoremstyle{definition}
\newtheorem{definition}[theorem]{Definition}
\newtheorem{remark}[theorem]{Remark}
\newtheorem{example}[theorem]{Example}

\newtheorem*{thmA}{Theorem A}
\newtheorem*{thmB}{Theorem B}
\newtheorem*{corC}{Corollary C}

\newcommand{\dd}{\mathfrak{d}}
\newcommand{\ddb}{\bar{\mathfrak{d}}}
\newcommand{\LL}{\mathbb{L}}
\newcommand{\BB}{\mathfrak{B}}
\newcommand{\sym}{\mathfrak{S}}
\newcommand{\onev}{\mathbf{1}}

\newcommand{\Val}{\operatorname{Val}}
\newcommand{\val}{\operatorname{val}}
\newcommand{\dinv}{\operatorname{dinv}}
\newcommand{\area}{\operatorname{area}}
\newcommand{\sortw}{\operatorname{sort}}
\newcommand{\Vall}{\operatorname{V}}
\newcommand{\rec}{\operatorname{rec}}
\newcommand{\LD}{\mathsf{LD}}
\newcommand{\qbin}[2]{\genfrac{[}{]}{0pt}{}{#1}{#2}_{q}}
\newcommand{\Mu}{M_u}
\newcommand{\Mv}{M_v}
\newcommand{\swap}{\mathsf{s}}
\newcommand{\Lam}{\Lambda}

\title[Symmetric slices of the valley Delta conjecture]{Exchange
identities and symmetric slices\\ of the valley Delta conjecture}

\author{Henry Shin}

\date{}


\begin{document}

\begin{abstract}
The valley Delta conjecture of Haglund, Remmel and Wilson predicts
that the symmetric function $\Delta'_{e_{n-k-1}}e_n$ equals a
generating function over labelled Dyck paths with $k$ decorated
contractible valleys.  Unlike the rise version, which is now a
theorem, the valley version remains open; indeed it is not even known
that its combinatorial side is symmetric.  We prove that the
coefficients of $t^{0}$ and $t^{1}$ in the valley generating function
are symmetric functions for all $n\ge1$ and $k\ge0$; equivalently,
the
fixed--diagonal--multiset slices of area at most one are symmetric.
The theorem follows from an adjacent exchange identity for scaffold
classes, which we prove in a strictly stronger form refined by the
numbers of undecorated rows carrying the labels $r,r+1$ between
consecutive rows with other labels.  The proof develops a transfer--operator
calculus in a $q$--deformed two--variable algebra generated by two
commuting half--twists.  In this algebra the exchange reduces to two
scalar symmetric--series identities for the operator $T=u\dd+v$.  We
also verify the refined identity computationally at area two over
extensive finite ranges and state the resulting general conjecture.
\end{abstract}

\maketitle
\enlargethispage{2pt}

\setcounter{tocdepth}{1}

\section{Introduction}\label{sec:intro}

\subsection{The Delta conjectures}\label{ssec:delta}

The shuffle theorem of Carlsson and Mellit \cite{CarlssonMellit}
expresses the bigraded Frobenius characteristic of the ring of
diagonal harmonics, $\nabla e_n$, as a weighted generating function
over labelled Dyck paths.  The Delta conjecture of Haglund, Remmel and
Wilson \cite{HRW} proposes a far--reaching generalization: for
$0\le k<n$,
\begin{equation}\label{eq:delta}
\Delta'_{e_{n-k-1}}e_n
\;=\;
\mathrm{Rise}_{n,k}(x;q,t)
\;=\;
\Val_{n,k}(x;q,t),
\end{equation}
where $\Delta'_f$ is the eigenoperator on modified Macdonald
polynomials reviewed in Section~\ref{ssec:background-delta}, and the
two right--hand sides are generating functions over labelled Dyck
paths carrying $k$ decorations: on \emph{rises} in the first case, on
\emph{contractible valleys} in the second.  The rise version is now a
theorem: D'Adderio and Mellit \cite{DAdderioMellit} proved its
compositional refinement, and Blasiak, Haiman, Morse, Pun and
Seelinger \cite{BHMPS} proved the extended version by entirely
different methods.

The valley version, by contrast, is open, and it carries independent
weight: it is the valley statistic, not the rise statistic, whose refinements match the conjectural module-theoretic picture.
Zabrocki \cite{Zabrocki} conjectured that $\Delta'_{e_{n-k-1}}e_n$ is
the Frobenius characteristic of a module of super--diagonal
coinvariants, and the conjectural monomial basis for this module
proposed in \cite{DIVW} predicts Hilbert--series refinements that
agree with the valley side of \eqref{eq:delta}, not with the rise
side (see the discussion in \cite{DAdderioIraci}).  What is known
about the valley version is limited: Qiu and Wilson \cite{QiuWilson}
proved its extended form at $t=0$ or $q=0$ via ordered multiset
partitions; D'Adderio and Iraci \cite{DAdderioIraci} proved the
Schr\"oder case $\langle\,\cdot\,,e_{n-d}h_d\rangle$; and Iraci and
Vanden Wyngaerd \cite{IVWpush,IVWsquare} showed that the valley Delta
conjecture implies its generalized and square versions.  As
D'Adderio and Iraci emphasize, away from these specializations even
the basic structural property is unknown: it is open whether the
combinatorial side $\Val_{n,k}(x;q,t)$ is a symmetric function at all
\cite{DAdderioIraci}.

\subsection{Main results}\label{ssec:main-results}

This paper proves the first positive--area symmetry result for the
valley side.  Write $\Val_{n,k}(x;q,t)=\sum_{m}
t^{\area(m)}\,\Vall_{n,k,m}(x;q)$, where the sum runs over
\emph{diagonal multisets} $m$ (the multiset of diagonals occupied by
the path) and $\Vall_{n,k,m}$ collects the paths with diagonal
multiset $m$; precise definitions are in
Section~\ref{ssec:background-paths}.  The multisets of area $0$ and
$1$ are unique ($m=0^n$ and, for $n\ge2$, $m=0^{n-1}1$), so the
$t^0$ and $t^1$ coefficients of $\Val_{n,k}$ are exactly the slices
$\Vall_{n,k,0^n}$ and $\Vall_{n,k,0^{n-1}1}$ (the latter vacuous for
$n=1$).

\begin{corC}
For all $n\ge 1$ and $k\ge 0$, the slices $\Vall_{n,k,0^{n}}$ and
(for $n\ge2$) $\Vall_{n,k,0^{n-1}1}$ are symmetric functions.
Equivalently, the coefficients of $t^{0}$ and of $t^{1}$ in
$\Val_{n,k}(x;q,t)$ are symmetric functions.
\end{corC}
This result is proved as Corollary~\ref{cor:slices} in
Section~\ref{sec:proofs}.

For area one this is, to our knowledge, the first symmetry statement
for any positive--area slice of the valley side; for area zero it
gives a new proof, independent of the ordered--multiset--partition
methods of \cite{QiuWilson,Rhoades}, of the symmetry implicit in the
$t=0$ case of the conjecture.  We emphasize the boundary of the
claim: we prove symmetry of these slices of the \emph{combinatorial}
side, and do not identify them with the corresponding coefficients of
$\Delta'_{e_{n-k-1}}e_n$, which would be the $t^0$ and $t^1$ cases of
the conjecture itself (see Remark~\ref{rem:QW}).

Corollary~C is obtained through the \emph{exchange} strategy.  A
formal series of bounded degree is symmetric if and only if it is
invariant under each adjacent transposition $x_r\leftrightarrow
x_{r+1}$; grouping the paths contributing to $\Vall_{n,k,m}$ according
to all of their data except what involves the labels $r,r+1$ (the resulting datum is called a \emph{scaffold} $\Xi$, recording the
relative order, diagonals, labels and decorations of the other rows) reduces the invariance to a two--variable identity for each
\emph{class} $(m,k,\Xi)$: writing $u,v$ for the variables tracking the
labels $r$ and $r+1$ and $\Phi^{(r)}_{m,k,\Xi}(u,v;q)$ for the class
generating function, one needs
\begin{equation}\label{eq:exchange-intro}
\Phi^{(r)}_{m,k,\Xi}(u,v;q)\;=\;\Phi^{(r)}_{m,k,\Xi}(v,u;q).
\end{equation}
This is the route by which Haglund, Haiman and Loehr proved the
symmetry of the LLT and Macdonald combinatorial formulas
\cite{HHL}.  We call identity \eqref{eq:exchange-intro}, required
for all $m,k,\Xi,r$, the \emph{scaffold exchange conjecture} and
formulate it precisely in Conjecture~\ref{conj:exchange}.  Decorations
obstruct the classical local exchange involutions: the swaps that
prove \cite{HHL}--type symmetry move rows across one another, while
the validity of a decoration depends on the row immediately to its
left.  The scaffold exchange conjecture is open in general, but
Proposition~\ref{prop:reduction} shows that it implies the symmetry
of every slice $\Vall_{n,k,m}$, hence of $\Val_{n,k}$ itself.

Our main theorem establishes the conjecture in every class of area at
most one, in a strictly stronger, refined form.  The rows of an
object whose labels avoid $\{r,r+1\}$ cut the remaining (\emph{active})
rows into \emph{gaps}; the refinement fixes, in addition to
$(m,k,\Xi)$, the number of \emph{undecorated} active rows in each gap.

\Needspace{14\baselineskip}
\begin{thmA}
Let $n\ge1$ and let $m$ be a size--$n$ diagonal multiset of area at
most one, i.e.\ $\sortw(m)\in\{0^n,\,0^{n-1}1\}$, the second case
occurring only for $n\ge2$.  For every $k\ge0$, $r\ge1$, and every scaffold $\Xi$ of
length $s$, and for every $g=(g_0,\dots,g_s)\in\mathbb{Z}^{s+1}$
interpreted as per--gap undecorated--active counts (the walls of
$\Xi$ cut the active rows into gaps $0,\dots,s$;
Definition~\ref{def:refined}), the refined class generating function
satisfies
\[
\Phi^{(r)}_{m,k,\Xi,g}(u,v;q)\;=\;\Phi^{(r)}_{m,k,\Xi,g}(v,u;q).
\]
\end{thmA}
This result is proved as Theorem~\ref{thm:main} in
Section~\ref{sec:proofs}.

\begin{thmB}
The scaffold exchange conjecture holds for every diagonal multiset of
area at most one.
\end{thmB}
This result is proved as Theorem~\ref{thm:conj} in
Section~\ref{sec:proofs}.

The refinement in Theorem~A is sharp in a precise sense.  Fixing
instead the \emph{total} number of active rows in each gap (equivalently, the absolute positions of the non--active rows) destroys the symmetry already for $m=0^4$, $k=1$, with a single
non--active row (Example~\ref{ex:sharp}); the failure is repaired
exactly when decorated rows are allowed to migrate within their gap
count, which is what counting \emph{undecorated} rows achieves.  The
refined statement appears to be a feature of the conjecture itself
rather than of low area: we have verified it at area two
exhaustively over finite ranges ($40{,}143$ class verifications,
$37{,}027$ distinct refined classes, in an even finer, per--diagonal
form) and conjecture it for all diagonal
multisets (Conjecture~\ref{conj:refined}).

\subsection{Method: a transfer calculus for decorated
rows}\label{ssec:method}

The proof is by an operator calculus that may be of independent
interest, since it gives a workable substitute for the missing
exchange involution in the presence of decorations.  After a
\emph{dictionary lemma} (Lemma~\ref{lem:dict}) converts the statistic
$\dinv$ and all validity constraints of a class into local data of a
word in six letter types, the generating function of a refined class
becomes a word in a small algebra of $q$--difference operators on
$\mathbb{Z}[q,q^{-1}][u,v,z]$, built from two commuting
\emph{half--twists}
\[
\dd=(1+qzv)\,\sigma_v,
\qquad
\ddb=(1+qzu)\,\sigma_u
\qquad(\sigma_u\colon u\mapsto qu,\ \ \sigma_v\colon v\mapsto qv),
\]
where $z$ marks decorations, and we write $\onev$ for the
constant polynomial $1$.  A single transfer operator
$T=u\dd+v$ realizes one undecorated active row together with its
optional trailing decorated companion; the operator
$\LL=\dd\ddb$ realizes an undecorated low non--active row together
with its decorated head cluster; and the diagonal--one row of the
$t^1$ slice enters through one extra letter.  The generators satisfy
quantum--torus relations
($\dd\Mu=\Mu\dd$, $\dd\Mv=q\Mv\dd$, and mirror images), $\LL$ is
central up to the scaling $\LL T=qT\LL$, and every refined class
collapses to an explicit scalar times one of
\[
\LL^{k}\,T^{m}\onev,
\qquad
uv\cdot\LL^{k}T^{m}\onev,
\qquad
\LL^{k}\bigl(u\,T^{g+M}\onev+v\,T^{g}\ddb\,T^{M}\onev\bigr).
\]
Symmetry of the first two families reduces to a Gaussian--binomial
identity (Theorem~\ref{thm:identityI}); symmetry of the third, Lemma~$\BB$ (Lemma~\ref{lem:B}), follows from an exact
first--order recursion in $g$.  We also show
(Proposition~\ref{prop:noR}) that no intertwining $R$--matrix for the
letter operators themselves can exist, even in the undecorated
single--diagonal case; the two scalar identities appear to be the
correct replacement.  For the diagonal set $\{0,1\}$, this identifies the precise
quantum--torus transfer structure needed for decorated rows.

The computational checks are independent of the proofs: the complete
chain (dictionary, operator identities, and the assembled formulas
with all their gates and gluings) has been verified by computer
against brute--force enumeration, $58{,}013$ refined--class
verifications ($53{,}054$ distinct classes) with exact agreement
(Section~\ref{ssec:verification}).

\subsection{Organization}\label{ssec:organization}

Section~\ref{sec:background} recalls the combinatorial model, states
the conjectures, and proves the reduction from exchange to symmetry.
Section~\ref{sec:dictionary} develops the two--letter dictionary and
presents the sharpness example.  Section~\ref{sec:algebra} introduces
the transfer algebra and proves the fundamental identity~(I).
Sections~\ref{sec:walls} and~\ref{sec:caseB} carry out the assembly:
walls, gates and gluing in Section~\ref{sec:walls}; the diagonal--one
active row and Lemma~$\BB$ in Section~\ref{sec:caseB}.
Section~\ref{sec:proofs} assembles the main theorems.
Section~\ref{sec:complements} contains the impossibility result for
letter--level $R$--matrices, the general refined conjecture with its
computational evidence, the verification protocol, and open problems.

\section{Background and the exchange strategy}\label{sec:background}

\subsection{Labelled Dyck paths, valleys, and
\texorpdfstring{$\dinv$}{dinv}}\label{ssec:background-paths}

We encode labelled Dyck paths by their area words.  Throughout,
$n\ge 1$ and $[n]=\{1,\dots,n\}$.

\begin{definition}\label{def:objects}
A \emph{labelled path} of size $n$ is a pair $(a,w)$ where
$a=(a_1,\dots,a_n)\in\mathbb{Z}_{\ge0}^{\,n}$ is an \emph{area word},
meaning $a_1=0$ and $a_{j+1}\le a_j+1$ for $1\le j<n$, and
$w=(w_1,\dots,w_n)\in\mathbb{Z}_{>0}^{\,n}$ is a \emph{label word}
satisfying the \emph{rise condition}: $w_j<w_{j+1}$ whenever
$a_{j+1}=a_j+1$.  We call $a_j$ the \emph{diagonal} of row $j$ and set
$\area(a):=\sum_j a_j$.  The \emph{diagonal multiset} of $(a,w)$ is
the multiset $m=\{\!\{a_1,\dots,a_n\}\!\}$; for a diagonal multiset
$m$ we write $\sortw(m)$ for its elements listed in weakly increasing
order.

Row $j$ is a \emph{contractible valley} if $j\ge2$ and either
$a_{j-1}>a_j$ (a \emph{drop} valley) or $a_{j-1}=a_j$ and
$w_{j-1}<w_j$ (a \emph{flat} valley); we write $\val(a,w)$ for the
set of contractible valleys.  The \emph{attack set} is
\[
A(a,w)=\{(i,j): i<j,\ a_i=a_j,\ w_i<w_j\}\ \cup\
        \{(i,j): i<j,\ a_i=a_j+1,\ w_i>w_j\},
\]
attacks of the first kind being \emph{primary} and of the second
\emph{secondary}.  A \emph{$k$--decorated} labelled path is a triple
$(a,w,S)$ with $S\subseteq\val(a,w)$, $|S|=k$, and its statistic is
\[
\dinv(a,w,S)\;:=\;\#\{(i,j)\in A(a,w):\ i\notin S\}\;-\;|S|.
\]
We write $\LD(n)^{\bullet k}$ for the set of $k$--decorated labelled
paths of size $n$.
\end{definition}

These conventions agree with those of
\cite{HRW,QiuWilson,DAdderioIraci}: in the lattice--path picture,
$a_j$ is the number of full cells between the path and the diagonal
in the $j$--th row, the rise condition is the column--strictness of
labels, decorated valleys may be contracted, and removing the attacks
emitted by decorated rows together with the correction $-|S|$
produces the valley--decorated $\dinv$.  The valley generating
function and its slices are
\begin{equation}\label{eq:Val}
\Val_{n,k}(x;q,t):=
\!\!\sum_{(a,w,S)\in\LD(n)^{\bullet k}}\!\!
q^{\dinv(a,w,S)}\,t^{\area(a)}\,x^{w},
\qquad
\Vall_{n,k,m}(x;q):=
\!\!\sum_{\substack{(a,w,S)\in\LD(n)^{\bullet k}\\
\{\!\{a\}\!\}=m}}\!\!
q^{\dinv}\,x^{w},
\end{equation}
where $x^w=x_{w_1}\cdots x_{w_n}$, so that
$\Val_{n,k}=\sum_m t^{\area(m)}\Vall_{n,k,m}$ with
$\area(m)=\sum_{c\in m}c$.  These are formal quasisymmetric
functions in the variables $x$ with coefficients in
$\mathbb{Z}[q,q^{-1},t]$ (respectively $\mathbb{Z}[q,q^{-1}]$ for
$\Vall_{n,k,m}$); when symmetry is proved, they lie in the
corresponding coefficient extension of $\Lambda$.

\begin{remark}\label{rem:lowarea}
An area word containing the value $c\ge 1$ must contain $c-1$ (the
word starts at $0$ and never climbs by more than one), so the only
diagonal multiset of area $0$ is $m=0^{n}$, and for $n\ge2$ the only
one of area $1$ is $m=0^{n-1}1$ (for $n=1$ no multiset of area $1$
exists, and the $t^{1}$ coefficient of $\Val_{1,k}$ is $0$).  Hence
the $t^{0}$ and $t^{1}$ coefficients of $\Val_{n,k}$ are the single
slices $\Vall_{n,k,0^n}$ and $\Vall_{n,k,0^{n-1}1}$, the latter read
as $0$ when $n=1$.
\end{remark}

\begin{remark}\label{rem:qsym}
Each $\Vall_{n,k,m}$ is a quasisymmetric function: validity, the
attack set, and the decorations all depend on the label word only
through its pattern of strict inequalities and equalities, so the
coefficient of $x_{i_1}^{\alpha_1}\cdots x_{i_\ell}^{\alpha_\ell}$
($i_1<\dots<i_\ell$) depends only on the composition
$(\alpha_1,\dots,\alpha_\ell)$.  Symmetry is the open property.
\end{remark}

\subsection{Delta operators and the valley
conjecture}\label{ssec:background-delta}

Let $\Lam$ be the algebra of symmetric functions over
$\mathbb{Q}(q,t)$ and $\{\tilde H_\mu\}$ the modified Macdonald basis;
see \cite{HaglundBook}.  For a symmetric function $f$, the operators
$\Delta_f,\Delta'_f$ act diagonally by
$\Delta_f\tilde H_\mu=f[B_\mu]\tilde H_\mu$ and
$\Delta'_f\tilde H_\mu=f[B_\mu-1]\tilde H_\mu$, where
$B_\mu=\sum_{(i,j)\in\mu}q^{j-1}t^{i-1}$.

\begin{conjecture}[valley Delta conjecture; Haglund--Remmel--Wilson
\cite{HRW}]\label{conj:valley}
For $0\le k<n$,
\(
\Delta'_{e_{n-k-1}}e_n=\Val_{n,k}(x;q,t).
\)
\end{conjecture}

The rise analogue of Conjecture~\ref{conj:valley} is a theorem
\cite{DAdderioMellit,BHMPS}.  For the valley version, the cases
$t=0$ and $q=0$ of the extended form are proved in \cite{QiuWilson}
(see also \cite{Rhoades,HRS} for the underlying ordered--set--%
partition combinatorics), the Schr\"oder case in
\cite{DAdderioIraci}, and the implications to the generalized and
square valley versions in \cite{IVWpush,IVWsquare}.  Since
Conjecture~\ref{conj:valley} would identify $\Val_{n,k}$ with a
symmetric function, the symmetry of $\Val_{n,k}$ is a necessary
consistency property; it is open in general \cite{DAdderioIraci}.
Symmetry of $\Val_{n,k}$ amounts to the symmetry of the coefficient
of each power of $t$, i.e.\ of the \emph{sum} of all slices
$\Vall_{n,k,m}$ over the multisets $m$ of a given area; the symmetry
of every individual slice is an a priori stronger property.  It is
the stronger property that is established here for $\area(m)\le 1$, where, the multiset of each area being unique
(Remark~\ref{rem:lowarea}), the two properties in fact coincide.

\subsection{Scaffolds, classes, and the exchange
reduction}\label{ssec:background-exchange}

Fix $r\ge 1$.  Labels in $\{r,r+1\}$ are \emph{active}; we call them
\emph{colors} $1$ and $2$.  All other labels are \emph{outside}.

\begin{definition}\label{def:scaffold}
Let $(a,w,S)\in\LD(n)^{\bullet k}$.  The \emph{scaffold} of
$(a,w,S)$ with respect to $r$ is the sequence
$\Xi=\bigl((d_t,\ell_t,\varepsilon_t)\bigr)_{t=1}^{s}$ recording, in
left--to--right order, the diagonal, label and decoration indicator
of the rows with outside labels; thus $d_t\in\mathbb{Z}_{\ge0}$,
$\ell_t\in\mathbb{Z}_{>0}\setminus\{r,r+1\}$, and
$\varepsilon_t\in\{0,1\}$, and we call any finite sequence of such
triples a \emph{scaffold} for $r$, whether or not it arises from an
object.  For fixed data $(m,k,\Xi)$ the
\emph{class} is the set of all $(a,w,S)\in\LD(n)^{\bullet k}$ with
diagonal multiset $m$ and scaffold $\Xi$, and its generating function
is
\[
\Phi^{(r)}_{m,k,\Xi}(u,v;q)\;:=\;
\sum_{(a,w,S)}\;
q^{\dinv(a,w,S)}\;
u^{\#\{j:\,w_j=r\}}\;
v^{\#\{j:\,w_j=r+1\}}.
\]
If no such object exists (for instance if $m$ is not realizable
as the diagonal multiset of a size--$n$ area word, or no object with
multiset $m$ has scaffold $\Xi$), the class is empty and
$\Phi^{(r)}_{m,k,\Xi}:=0$.
\end{definition}

\begin{conjecture}[scaffold exchange]
\label{conj:exchange}
For every $n\ge1$, $k\ge0$, and $r\ge1$, every diagonal multiset
$m$ of size $n$, and every scaffold
$\Xi$,
\[
\Phi^{(r)}_{m,k,\Xi}(u,v;q)\;=\;\Phi^{(r)}_{m,k,\Xi}(v,u;q).
\]
\end{conjecture}

\begin{remark}\label{rem:D}
One may also fix, in addition, the number $D$ of attacks among outside
rows with undecorated emitter.  This datum is
redundant: for outside rows at scaffold positions $t<t'$ the attack
condition reads $(d_t=d_{t'}\wedge\ell_t<\ell_{t'})$ or
$(d_t=d_{t'}+1\wedge\ell_t>\ell_{t'})$, and the pair is counted iff
$\varepsilon_t=0$; every ingredient is a function of $\Xi$ alone.  We
write $D=D(\Xi)$ for this constant.
\end{remark}

The reduction from exchange to symmetry is the same mechanism by
which Haglund, Haiman and Loehr proved the symmetry of the
combinatorial Macdonald and LLT formulas \cite{HHL}.  We include the
short proof for completeness.

\begin{proposition}\label{prop:reduction}
Fix $n\ge1$, $k\ge0$, and a size--$n$ diagonal multiset $m$.  If
Conjecture~\ref{conj:exchange} holds for $(m,k,\Xi,r)$ for every
scaffold $\Xi$ and every $r\ge1$, then $\Vall_{n,k,m}$ is a symmetric
function.
\end{proposition}

\begin{proof}
$\Vall_{n,k,m}$ is a homogeneous formal series of degree $n$ in
$x_1,x_2,\dots$; it suffices to prove invariance under each adjacent
transposition $x_r\leftrightarrow x_{r+1}$, since these generate all
permutations of the variables and a bounded--degree series invariant
under all of them has coefficients constant on orbits of exponent
vectors.  Fix $r$ and group the objects by their scaffold:
\[
\Vall_{n,k,m}
=\sum_{\Xi}\Bigl(\textstyle\prod_{t=1}^{s}x_{\ell_t}\Bigr)\,
\Phi^{(r)}_{m,k,\Xi}(x_r,x_{r+1};q),
\]
which is simply the partition of the objects contributing to
$\Vall_{n,k,m}$ according to their scaffold: each object determines
its scaffold, every label of an object is either an outside label
$\ell_t$ (contributing $x_{\ell_t}$ to the prefactor) or one
of $r,r+1$ (counted by the exponents of $u,v$ in
$\Phi^{(r)}_{m,k,\Xi}$), and the weight $q^{\dinv}$ is carried by
the coefficients of $\Phi^{(r)}_{m,k,\Xi}$.  The scaffold sum is
locally finite: a fixed monomial $x^\alpha$ receives
contributions only from the finitely many objects whose labels lie
in the support of $\alpha$, hence from finitely many scaffolds, so the regrouping, and the termwise application of the swap below,
are valid coefficientwise.  The prefactor
$\prod_t x_{\ell_t}$ involves no $x_r,x_{r+1}$, so the swap
$x_r\leftrightarrow x_{r+1}$ fixes each summand by the assumed
exchange identity.
\end{proof}

The wall--free case $\Xi=\varnothing$, $m=0^n$ already contains the
basic decorated two--letter kernel; it appears below as
Corollary~\ref{cor:empty}.  Outside rows create the new difficulty:
the validity of a decoration is a condition on the immediate left
neighbour of the decorated row, which local exchange moves do not
preserve.  Theorem~\ref{thm:conj} establishes
Conjecture~\ref{conj:exchange} for all $m$ with $\area(m)\le1$, and
Theorem~\ref{thm:main} proves more: the exchange survives the finer
fibering of Definition~\ref{def:refined} below.

\section{The two--letter dictionary}\label{sec:dictionary}

For the rest of the paper we fix $r\ge1$ and a diagonal multiset $m$
with $\sortw(m)\in\{0^n,0^{n-1}1\}$, and we study the classes
$(m,k,\Xi)$.  We refer to outside rows as \emph{walls}; a wall with
$\ell_t<r$ has type $L$ (\emph{low}), one with $\ell_t>r+1$ type $H$
(\emph{high}); these are the only two possibilities, since scaffold labels
avoid $\{r,r+1\}$ (Definition~\ref{def:scaffold}).  Active rows number $N:=n-s$.  We first discard the following
immediate incompatibilities: if $m$ is not realizable by an area word of
size $n$ (so the case $0^{n-1}1$ occurs only for $n\ge2$,
Remark~\ref{rem:lowarea}), the multiset of wall diagonals is not
contained in $m$, the flags satisfy
$\#\{t:\varepsilon_t=1\}>k$, the number
$k-\#\{t:\varepsilon_t=1\}$ of decorated active rows exceeds $N$,
or $s>n$, then the class is empty and
there is nothing to prove.  We call the remaining data compatible, but a
compatible datum may still define an empty refined class because of the
local rise, valley, gate, and glue conditions recorded below.  The
structure of a compatible class then splits into three regimes:

\smallskip
\noindent\textbf{Flat case} ($m=0^n$): all rows lie on diagonal $0$.

\smallskip
\noindent\textbf{Case A} ($m=0^{n-1}1$, diagonal--$1$ row outside):
exactly one scaffold entry has $d_{t^*}=1$; we call this wall $W^*$,
with label $\ell^*$ of type $\tau^*\in\{L,H\}$.  All active rows lie
on diagonal $0$.

\smallskip
\noindent\textbf{Case B} ($m=0^{n-1}1$, diagonal--$1$ row active):
all scaffold entries have $d_t=0$, and exactly one active row,
denoted $P$, lies on diagonal $1$.

\smallskip
In Case~A we assume moreover $\varepsilon_{t^*}=0$: a decorated
$W^*$ makes the class empty, by (V5) of Lemma~\ref{lem:dict} below.

\smallskip
An object of the class is determined by: the interleaving of the $s$
walls (in scaffold order) with the $N$ active rows; a color
$c\in\{1,2\}$ for each active row; the set of decorated active rows,
of size $\kappa:=k-\#\{t:\varepsilon_t=1\}$; and, in Case~B, the
choice of $P$.  In Case~B, moreover, $\kappa\le N-1$ is necessary
for a nonempty class: $P$ is active and never decorated (its pred
lies on diagonal $0$, so $P$ sees a rise and is not a valley,
Definition~\ref{def:objects}); equivalently, every refined tuple
with $|g|=0$ vanishes in Case~B.

\begin{definition}[gaps; refined classes]\label{def:refined}
The walls divide the active rows into \emph{gaps} $0,1,\dots,s$, gap
$t$ lying between wall $t$ and wall $t+1$ (gap $0$ before wall $1$,
gap $s$ after wall $s$).  For an object $O$ let $g_t(O)$ be the
number of \emph{undecorated} active rows in gap $t$; in Case~B the
row $P$ is counted ($P$ is never decorated, by
Lemma~\ref{lem:dict}(V2)).  For an integer vector
$g=(g_0,\dots,g_s)\in\mathbb{Z}^{s+1}$ with $g_t\ge0$ for all $t$
and $\sum_tg_t=N-\kappa$, the \emph{refined class} $(m,k,\Xi,g)$ consists
of the objects with $g_t(O)=g_t$ for all $t$, and its generating
function is the class sum restricted to those objects:
\[
\Phi^{(r)}_{m,k,\Xi,g}(u,v;q):=
\sum_{O\in(m,k,\Xi,g)}q^{\dinv(O)}
 u^{\#\{j:\,w_j(O)=r\}}
 v^{\#\{j:\,w_j(O)=r+1\}}.
\]
Thus $\Phi^{(r)}_{m,k,\Xi}=\sum_g\Phi^{(r)}_{m,k,\Xi,g}$.  For any other
integer vector $g$ (wrong length, a negative entry, or
$\sum_tg_t\ne N-\kappa$), we set $\Phi^{(r)}_{m,k,\Xi,g}:=0$.
\end{definition}

The next lemma converts $\dinv$ and all validity constraints into a
word--local \emph{dictionary}.  For a row $x$ we write ``before $x$''
for rows at smaller positions and ``pred'' for the row immediately
preceding $x$ in the word.

\begin{lemma}[Dictionary]\label{lem:dict}
Let $O$ be an object of the class $(m,k,\Xi)$,
$\sortw(m)\in\{0^n,0^{n-1}1\}$.  Then
\begin{equation}\label{eq:dict}
\dinv(O)\;=\;D(\Xi)\;-\;k\;+\;\operatorname{brk}(O),
\qquad
\operatorname{brk}(O):=\sum_{x\ \mathrm{row\ of}\ O}\rec(x),
\end{equation}
where the receiver weights $\rec(x)$ are as follows.  (Here and in
(V1)--(V5), ``active'' and ``wall'' refer to rows on diagonal $0$
unless $W^*$ or $P$ is named explicitly: $W^*$, on diagonal $1$, is
never counted among the $L$-- or $H$--walls below, and in Case~B the
row $P$, though active, enters only through the terms naming it.
Throughout, $[\mathcal S]$ denotes the Iverson bracket: $1$ if the
statement $\mathcal S$ holds and $0$ otherwise.)
\begin{itemize}
\item active row $x$ on diagonal $0$ of color $c$:
$\rec(x)=\#\{\text{undecorated }L\text{--walls before }x\}
+[c{=}2]\cdot\#\{\text{undecorated color--}1\text{ actives before }x\}
+[\,W^*\text{ before }x\ \wedge\ \tau^*{=}H\,]
+[\,c{=}1\ \wedge\ P\text{ before }x\ \wedge\ P\text{ has color }2\,]$;
\item wall $t$ on diagonal $0$:
$\rec=[\ell_t{>}r{+}1]\cdot\#\{\text{undecorated actives before}\}
+[\ell_t{<}r]\cdot[\,P\text{ before}\,]$;
\item $W^*$ and $P$: $\rec=0$.
\end{itemize}
Moreover the valid objects are exactly the tuples
(interleaving, colors, decorations, $P$--position) satisfying:
\begin{itemize}
\item[(V1)] (rise at $W^*$, Case A) pred of $W^*$ exists and has
label $<\ell^*$: if pred is active this holds iff $\tau^*=H$
(irrespective of color and decoration); if pred is wall $t^*-1$ it
holds iff $\ell_{t^*-1}<\ell^*$;
\item[(V2)] (rise at $P$, Case B) pred of $P$ is an $L$--wall (then
both colors of $P$ are allowed) or an active of color $1$ (then $P$
has color $2$); pred a high wall, an active of color $2$, or the
start of the word is forbidden.  $P$ is never a valley, hence never
decorated.
\item[(V3)] (decorated active $x$) pred of $x$ is an $L$--wall (any
flag), $W^*$, or $P$ (drop or low flat predecessor), or pred is an
active of color $1$ and $x$ has color $2$;
\item[(V4)] (decorated wall $t$, $d_t=0$) pred is: an active row
(possible only if $\ell_t>r+1$), or $W^*$ or $P$ (drop), or wall
$t-1$ with $\ell_{t-1}<\ell_t$;
\item[(V5)] a flag $\varepsilon=1$ on the entry $W^*$ makes the
class empty ($W^*$ terminates a rise, hence is never a valley); the
first row of the word is never a valley.
\end{itemize}
\end{lemma}

\begin{proof}
By definition $\dinv=|A^S|-k$ with $A^S=\{(i,j)\in A:i\notin S\}$.
Splitting $|A^S|$ according to whether both rows are outside or not
gives the constant $D(\Xi)$ of Remark~\ref{rem:D} plus the sum, over
all rows $x$, of the number of attacks received by $x$ from earlier
rows with unscreened emitter, at least one of the two rows being
active; it remains to enumerate the attack conditions for each
ordered pair of row types, recording them at the receiver.  Recall
that active labels are $r,r+1$ and outside labels avoid both.

\emph{Active $\to$ active, both on diagonal $0$.}  Primary, and
$w_i<w_j$ iff $(c_i,c_j)=(1,2)$.  Counted iff the emitter is
undecorated: the second active receiver term.

\emph{$L$--wall $\to$ active and $H$--wall $\to$ active (diagonal
$0$).}  Primary; $\ell<w_j$ holds always for $\ell<r$ and never for
$\ell>r+1$.  Counted iff the wall is undecorated: the first term.

\emph{Active $\to$ wall (diagonal $0$).}  Primary; $w_i<\ell_t$ holds
iff $\ell_t>r+1$.  Counted iff the active is undecorated: the
$H$--wall receiver term.  (For $\ell_t<r$ there is no attack.)

\emph{$W^*$ versus diagonal--$0$ rows (Case A).}  If $W^*$ is
earlier: a secondary attack $(W^*,j)$ requires $\ell^*>w_j$, i.e.\
$\tau^*=H$ when $j$ is active (the third active receiver term; $W^*$
is never decorated by (V5), so the attack is always counted) and
$\ell^*>\ell_t$ when $j$ is a wall (an outside--outside pair, inside
$D(\Xi)$).  If $W^*$ is later: a primary attack needs equal
diagonals and a secondary one needs $a_i=2$; both are impossible, so
$W^*$ receives nothing.

\emph{$P$ versus diagonal--$0$ rows (Case B).}  If $P$ is earlier:
the secondary condition $w_P>w_j$ holds for an active $j$ iff
$(c_P,c_j)=(2,1)$ (the fourth active receiver term; $P$ is never
decorated) and for a wall iff $\ell_t<r$ (the $L$--wall receiver
term; both colors of $P$ exceed every low label, so the term is
color--blind).  If $P$ is later, both attack types are impossible as
for $W^*$, so $P$ receives nothing.

Summing the receiver weights over all rows reproduces
$|A^S|-D(\Xi)$, proving~\eqref{eq:dict}.

For validity: the rise condition occurs exactly at $W^*$ (Case~A) or
at $P$ (Case~B), giving (V1) and (V2) by inspection of the label
order $\ell<r<r+1<\ell'$ for $L$-- and $H$--labels $\ell,\ell'$; the
forced color in (V2) is because $w_{\mathrm{pred}}<w_P$ with both
labels in $\{r,r+1\}$ forces $(r,r+1)$.  A decorated row must be a
contractible valley; applying the valley definition to each
predecessor type yields (V3)--(V4) by the same label--order
inspection (a drop predecessor, i.e.\ $W^*$ or $P$, always
validates; a flat predecessor validates iff its label is smaller).
Finally $W^*$ and $P$ terminate the unique rise, so neither is a
valley, and the first row has no predecessor: (V5).
\end{proof}

Two consequences of Lemma~\ref{lem:dict} are used constantly.
First, \emph{color--blindness} of validity outside the
active--active and $P$ interactions: every condition in (V1)--(V5)
involving an active predecessor either holds for both colors or for
neither, except the interior decoration rule ``pred of color $1$,
self of color $2$'' in (V3) and the rule (V2) at $P$.  Second, the
receiver weight of a row never depends on the decoration status of
the row itself, only on the decoration of \emph{earlier} rows:
screening is an emitter--side phenomenon.

We close this section with the example showing that the refinement
of Definition~\ref{def:refined} is the correct one.

\begin{example}[sharpness of the refinement]\label{ex:sharp}
Take $n=4$, $m=0^4$, $r=1$, $k=1$, and
$\Xi=\bigl((0,3,0)\bigr)$: one undecorated high wall, labelled $3$,
so the active letters are $1,2$ and the single decoration lies on an
active valley.  Refining by the \emph{total} number of active rows in
each gap (equivalently, by the absolute position of the wall) produces asymmetric pieces: writing $[\,u^\alpha v^\beta\,]\Phi$ for
coefficients,
\[
\begin{array}{c@{\qquad}c@{\qquad}c}
\toprule
\text{totals }(t_0,t_1) & [\,u^2v\,]\Phi & [\,uv^2\,]\Phi\\
\midrule
(0,3) & 1+q & 1+q\\
(1,2) & q^{2} & q\\
(2,1) & q & q^{2}\\
(3,0) & q^{2}+q^{3} & q^{2}+q^{3}\\
\bottomrule
\end{array}
\qquad\quad
\begin{array}{c@{\qquad}c@{\qquad}c}
\toprule
\text{undec.\ }(g_0,g_1) & [\,u^2v\,]\Phi & [\,uv^2\,]\Phi\\
\midrule
(0,2) & 1+q & 1+q\\
(1,1) & q+q^{2} & q+q^{2}\\
(2,0) & q^{2}+q^{3} & q^{2}+q^{3}\\
\bottomrule
\end{array}
\]
The totals $(1,2)$ and $(2,1)$ are individually asymmetric; their
union is the refined class $g=(1,1)$ of
Definition~\ref{def:refined}, which is symmetric, in accordance with
Theorem~\ref{thm:main}.  (The decorated active migrates across the
wall between the two total--vectors while the undecorated counts stay
fixed.)  The values also illustrate the normal form of
Proposition~\ref{prop:flatA}: here $\Phi_{g}=
q^{D(\Xi)-k}\,q^{g_0}\cdot(q+q^{2})(u^{2}v+uv^{2})
= q^{g_0-1}(q+q^2)(u^2v+uv^2)$, since the $z$--linear part of
$T^{2}\onev$ equals $(q+q^{2})z\,(u^{2}v+uv^{2})$.
\end{example}

\section{The transfer algebra and the fundamental
identity}\label{sec:algebra}

We work in the ring $R:=\mathbb{Z}[q,q^{-1}][u,v,z]$, and
$\onev$ denotes the constant polynomial $1\in R$.  For a polynomial
$F\in R$ let $M_F$ denote the operator of multiplication by $F$; in
particular $\Mu,\Mv$ are the multiplications by $u$ and by $v$.  Let
$\sigma_u,\sigma_v$ be the substitutions
$u\mapsto qu$, $v\mapsto qv$, and set
\[
\dd:=(1+qzv)\,\sigma_v,\qquad
\ddb:=(1+qzu)\,\sigma_u,\qquad
Q:=\sigma_u\sigma_v,\qquad
\LL:=\dd\,\ddb,
\]
\[
T:=u\,\dd+v
\;=\;\Mu\dd+\Mv,
\qquad
\bar T:=v\,\ddb+u,
\qquad
T^{\langle\lambda\rangle}:=\lambda\, u\,\dd+v.
\]
We also write $J:=(1+zu)(1+zv)$ and
$J^{(q)}:=(1+qzu)(1+qzv)$, so that $\LL=J^{(q)}Q$ as operators (the
substitutions in $\dd\ddb$ first commute past the prefactors:
$\sigma_v(1+qzu)=(1+qzu)\sigma_v$ and
$(1+qzv)\sigma_v\,(1+qzu)\sigma_u = (1+qzv)(1+qzu)\sigma_v\sigma_u$).
We use the standard $q$--Pochhammer symbol
$(A;q)_a:=\prod_{i=0}^{a-1}(1-Aq^{i})$, so that
$(-qzv;q)_a=\prod_{i=1}^{a}(1+q^{i}zv)$, and the Gaussian binomial
$\qbin{m}{a}:=(q;q)_m/\bigl((q;q)_a\,(q;q)_{m-a}\bigr)$ for
$0\le a\le m$, set to $0$ otherwise.
Finally $\swap$ denotes the involution of $R$ exchanging
$u\leftrightarrow v$, and
$\sym:=\{F\in R:\swap F=F\}$.  Both gradings of $R$ (total degree in
$u,v$, and degree in $z$) are preserved by $\swap$, so every
bigraded component of an element of $\sym$ lies in $\sym$; we use
this silently when extracting coefficients.

\begin{lemma}[commutation relations]\label{lem:comm}
As operators on $R$:
\begin{enumerate}
\item $\dd\ddb=\ddb\dd$, and $\swap\,\dd=\ddb\,\swap$,
$\swap\,Q=Q\,\swap$, $\swap\,\LL=\LL\,\swap$,
$\swap\,T=\bar T\,\swap$;
\item $\dd \Mu=\Mu\dd$,\quad $\dd \Mv=q\,\Mv\dd$,\quad
$\ddb \Mv=\Mv\ddb$,\quad $\ddb \Mu=q\,\Mu\ddb$;
\item $(\Mu\dd)\,\Mv \;=\; q\,\Mv\,(\Mu\dd)$, hence the
$q$--binomial theorem
\begin{equation}\label{eq:qbinT}
T^{m}\;=\;\sum_{a=0}^{m}\qbin{m}{a}\,u^{a}v^{m-a}\,\dd^{a},
\qquad
\dd^{a}\onev=\prod_{i=1}^{a}(1+q^{i}zv)=(-qzv;q)_a;
\end{equation}
\item $\LL \Mu = q\,\Mu \LL$ and $\LL \Mv=q\,\Mv\LL$; consequently
$\LL T = q\, T\LL$ and, for all $a,k\ge 0$,
\begin{equation}\label{eq:Lpull}
T^{a}\,\LL^{k}\;=\;q^{-ak}\,\LL^{k}\,T^{a};
\end{equation}
\item $T\,\ddb=\ddb\,T^{\langle 1/q\rangle}$ and
$\ddb\,T=T^{\langle q\rangle}\ddb$;
\item $T(\Mv f)=\Mv\,T^{\langle q\rangle}f$ and hence
$T^{g}(\Mv f)=\Mv (T^{\langle q\rangle})^{g}f$; also
$T\Mu=\Mu T$;
\item $\LL$ preserves $\sym$, and more generally
$\swap\bigl(\LL^{k}F\bigr)=\LL^{k}\,\swap F$.
\end{enumerate}
\end{lemma}

\begin{proof}
(1) The two factors of $\dd$ (resp.\ $\ddb$) involve only $v$
(resp.\ $u$), so $\dd$ and $\ddb$ commute; conjugation by $\swap$
exchanges the two by symmetry of the definitions, and $Q$, $\LL$,
$T\leftrightarrow\bar T$ follow.
(2) $\dd(uf)=(1+qzv)\sigma_v(uf)=u\,\dd f$ and
$\dd(vf)=(1+qzv)(qv)\sigma_vf=qv\,\dd f$; the $\ddb$ relations are
mirror images.
(3) $(\Mu\dd)(\Mv f)=u\,\dd(vf)=quv\,\dd f=q\Mv(\Mu\dd)f$ by (2).
Set $A=\Mu\dd$ and $B=\Mv$.  Since $AB=qBA$, the Gaussian binomial
theorem in normal order gives
$(A+B)^m=\sum_{a=0}^m\qbin{m}{a}B^{m-a}A^a$.  Now
$A^{a}=(\Mu\dd)^{a}=u^{a}\dd^{a}$ by the first relation in (2), so
$B^{m-a}A^a=v^{m-a}u^{a}\dd^a=u^av^{m-a}\dd^a$.  This is
\eqref{eq:qbinT}.  Finally
$\dd^a\onev=\prod_{i=1}^a(1+q^izv)$ since each application of $\dd$
scales the $v$'s of the current prefactor by $q$ and appends one
factor $(1+qzv)$.
(4) $\LL(uf)=\dd\ddb(uf)=\dd(qu\,\ddb f)=qu\,\LL f$ by (2), and
symmetrically for $v$; then
$\LL T=\LL\Mu\dd+\LL\Mv=q\Mu\dd\LL+q\Mv\LL=qT\LL$ (using
$\LL\dd=\dd\LL$), and \eqref{eq:Lpull} follows by iteration.
(5) $T\ddb=\Mu\dd\ddb+\Mv\ddb$ and
$\ddb T^{\langle 1/q\rangle}=\ddb(q^{-1}\Mu\dd+\Mv)
=q^{-1}(q\Mu\ddb)\dd+\Mv\ddb=\Mu\dd\ddb+\Mv\ddb$ by (2) and (1);
the second relation is the same computation read backwards.
(6) $T(vf)=u\dd(vf)+v^2f=u(qv)\dd f+v\cdot vf
=v\,(qu\dd+v)f=v\,T^{\langle q\rangle}f$, and $T\Mu=\Mu T$ by (2).
(7) $\LL=J^{(q)}Q$ with $J^{(q)}\in\sym$ a symmetric multiplier and
$Q$ commuting with $\swap$; hence $\swap\LL=\LL\swap$ by (1).
\end{proof}

\begin{theorem}[Identity (I)]\label{thm:identityI}
For every $m\ge 0$,
\[
T^{m}\onev\;=\;\sum_{a=0}^{m}\qbin{m}{a}\,u^{a}v^{m-a}\,(-qzv;q)_a
\;\in\;\sym.
\]
\end{theorem}

\begin{proof}
The displayed formula is~\eqref{eq:qbinT} applied to $\onev$.  Using
$[z^{\kappa}](-qzv;q)_a
= e_{\kappa}(q,q^2,\dots,q^{a})\,v^{\kappa}
= q^{\binom{\kappa+1}{2}}\qbin{a}{\kappa}v^{\kappa}$,
the coefficient of $z^{\kappa}$ in $T^m\onev$ equals
\[
q^{\binom{\kappa+1}{2}}\sum_{a=\kappa}^{m}
\qbin{m}{a}\qbin{a}{\kappa}\;u^{a}\,v^{\,m-a+\kappa}.
\]
Symmetry in $(u,v)$ of this polynomial amounts to the identity
\begin{equation}\label{eq:subsub}
\qbin{m}{a}\qbin{a}{\kappa}
\;=\;
\qbin{m}{\,m+\kappa-a\,}\qbin{\,m+\kappa-a\,}{\kappa},
\qquad \kappa\le a\le m,
\end{equation}
the right side being the coefficient of
$u^{\,m+\kappa-a}v^{a}$ (note $\kappa\le m+\kappa-a\le m$ exactly
when $\kappa\le a\le m$, so the ranges match).  By the
$q$--subset--of--subset identity
$\qbin{m}{b}\qbin{b}{\kappa}=\qbin{m}{\kappa}\qbin{m-\kappa}{b-\kappa}$,
the left side of~\eqref{eq:subsub} equals
$\qbin{m}{\kappa}\qbin{m-\kappa}{a-\kappa}$ and the right side equals
$\qbin{m}{\kappa}\qbin{m-\kappa}{m-a}$; these agree by the symmetry
$\qbin{n}{j}=\qbin{n}{n-j}$ of Gaussian binomials, since
$(a-\kappa)+(m-a)=m-\kappa$.
\end{proof}

\begin{corollary}\label{cor:TbarT}
For all $N\ge 0$, $\ \bar T\,T^{N}\onev=T^{N+1}\onev$.
\end{corollary}

\begin{proof}
By Lemma~\ref{lem:comm}(1), $\bar T\swap=\swap T$.  Applying both
sides to $T^N\onev\in\sym$ (Theorem~\ref{thm:identityI}) gives
$\bar T\,T^N\onev=\bar T\,\swap(T^N\onev)
=\swap\bigl(T^{N+1}\onev\bigr)=T^{N+1}\onev$, the last step again by
Theorem~\ref{thm:identityI}.
\end{proof}
\section{Walls, gates, and the assembly in the flat and outside
cases}\label{sec:walls}

We now express each refined class generating function as an explicit
word in the algebra of Section~\ref{sec:algebra} applied to $\onev$.
The mechanism is an emitter--side reattribution of the dictionary: in
$\operatorname{brk}=\sum_x\rec(x)$, each unscreened attack
contributing to $\rec(x)$ is charged instead to its \emph{emitter},
which implements it as a substitution applied to the part of the word
to the right of the emitter.  Concretely, in a product
$\mathcal{O}_{x_1}\mathcal{O}_{x_2}\cdots\mathcal{O}_{x_n}\onev$
(letters of the object read left to right, composition acting on the
right), a substitution $\sigma_v$ performed by the operator of letter
$x_i$ multiplies by $q$ the $v$--variable of \emph{every} letter
$x_j$, $j>i$, decorated or not, which is exactly the statement
``$x_i$ attacks every later row of color $2$''; similarly
$\sigma_u$, and $Q=\sigma_u\sigma_v$ for ``attacks every later
active row''.  Attacks received by later \emph{walls} carry no
variable and are charged to the emitter as explicit scalars.  In
Case~B the diagonal--one row $P$ is exempt throughout: no
diagonal--$0$ row attacks $P$ (Lemma~\ref{lem:dict}), and $P$'s
content is carried by markers transparent to every substitution
(Section~\ref{sec:caseB}), so ``active'' in the emitter statements
of this section refers to the diagonal--$0$ actives.

We mark an undecorated active row of color $1$ by $u$, of color $2$
by $v$, and decorated actives by $zu,zv$.  For a gap $t$ let
\[
h_t:=\#\{t'>t:\ \ell_{t'}>r+1,\ d_{t'}=0\},
\qquad
\Lambda_t:=\#\{t'>t:\ \ell_{t'}<r,\ d_{t'}=0\}
\]
count the high, respectively low, diagonal--$0$ walls to the right of
gap $t$, decorated or not, since by Lemma~\ref{lem:dict} the
receiver's decoration is irrelevant.

\begin{lemma}[head clusters]\label{lem:cluster}
Immediately after a row whose successor may be decorated (an
$L$--wall of either flag, $W^*$, or, in Case~B, the row $P$), the
maximal initial run of decorated actives (the \emph{head cluster};
for $P$, its \emph{dress}) is one of
$\varnothing,\{1^*\},\{2^*\},\{1^*2^*\}$, with content generating
function $J=(1+zu)(1+zv)$ and no internal attack weights (decorated
rows emit nothing).  After a high wall, or at the start of the word,
no decorated active may appear before the first undecorated one.
\end{lemma}

\begin{proof}
By (V3), after an $L$--wall, after $W^*$, or after $P$ both $1^*$ and
$2^*$ are valid; a $2^*$ may also follow a decorated $1^*$, but no
decorated row may follow a $2^*$ (the label $r+1$ admits no ascent
within the active alphabet, and an active pred is never a drop).  Two
consecutive $1^*$'s are excluded since $(r,r)$ is not an ascent.
After a high wall, or with no pred at all, (V3) and (V5) exclude
every decorated active.
\end{proof}

\begin{lemma}[gap steps]\label{lem:gapsteps}
Cut the sequence of active rows of a gap $t$ immediately before
each undecorated active; in Case~B, first remove the row $P$, its
dress (Lemma~\ref{lem:cluster}), and the active immediately
preceding $P$ whenever $P$'s pred is an active (necessarily an
undecorated $1$ or a decorated $1^*$, by Lemma~\ref{lem:dict});
together these removed rows form the \emph{active part} of the
$P$--block (Section~\ref{sec:caseB}).  This is the only
$P$--related material removed in this lemma: host--wall slots and any
glue chain whose demand is satisfied by $P$ are treated later in
Lemma~\ref{lem:Pblocks}.  Apart from a possible
initial segment of decorated rows (the head cluster of
Lemma~\ref{lem:cluster}, treated with its host), each segment, called a \emph{compound step}, is of exactly one of three kinds:
a single undecorated $1$; a single undecorated $2$; or an undecorated
$1$ immediately followed by a decorated $2$.  The operator of one
compound step is $q^{h_t}\,T$, the three kinds being the
three terms of
$q^{h_t}T=q^{h_t}\bigl(u\,\sigma_v+qz\,uv\,\sigma_v+v\bigr)$.
\end{lemma}

\begin{proof}
Every row of a segment after its first is decorated, with pred the
preceding row of the same segment: the cuts precede the undecorated
rows, and in Case~B the row immediately following the removed
$P$--block is undecorated, by maximality of $P$'s dress
(Lemma~\ref{lem:cluster}).  By (V3) such
a decorated row has color $2$ with pred an active of color $1$.
That pred is undecorated: it is either the first row of its segment,
undecorated by construction, or itself a decorated row with an
active pred, which (V3) would color $2$, not $1$.  A second
decorated row in a segment is therefore impossible (its pred would
be the first decorated $2$, of the wrong color, and $(r+1,r+1)$ is
not an ascent).  This proves the segment structure.  For the
weights: an undecorated $1$ contributes content $u$, attacks every
later active of color $2$ (the substitution $\sigma_v$, which also
multiplies the $zv$ of its own companion decorated $2$, accounting
for the attack inside the compound and producing the factor $qzv$ in
the middle term) and attacks every later high diagonal--$0$ wall:
the scalar $q^{h_t}$.  An undecorated $2$ contributes $v$, attacks no
later active (a primary attack needs a strictly larger label at the
receiver), and attacks the later high walls: again $q^{h_t}$.  A
decorated active emits nothing.  Summing the three kinds gives
$q^{h_t}\bigl(u(1+qzv)\sigma_v+v\bigr)=q^{h_t}T$.
\end{proof}

In the next two lemmas, a \emph{gate} is a Boolean local admissibility
condition on the current expansion branch.  A failed gate contributes
zero.  A \emph{glue demand} is a local condition requiring the decorated
wall to use the row immediately to its left as predecessor; when the
wall immediately to the right makes such a demand, the head-cluster
multiplier $J$ of the row or wall to its left is suppressed as an entire
factor, because that cluster is forced to be empty.  These operations
partition the local expansion branch by branch; they are not additional
components of the fixed refined datum $g$.

\begin{lemma}[wall operators, gates, and gluing]\label{lem:walls}
In the operator word of a refined class, wall $t$ (together with
its head cluster, if any) contributes:
\begin{enumerate}
\item \emph{undecorated $L$--wall:} the operator
$Q\circ M_J=J^{(q)}Q=\LL$;
\item \emph{undecorated $H$--wall:} the identity (it emits nothing
to later actives and hosts no cluster; its receptions are the
$q^{h}$--scalars of Lemma~\ref{lem:gapsteps});
\item \emph{decorated $H$--wall:} the identity, subject to the
\emph{gate}: the refined tuple contributes only if
$g_{t-1}\ge 1$, or $t\ge 2$ and ($d_{t-1}=1$ or
$\ell_{t-1}<\ell_t$);
\item \emph{$W^*$ with $\tau^*=H$ (Case A):} the operator $\LL$
(recall $\varepsilon_{t^*}=0$: a decorated $W^*$ gives the empty
class, Lemma~\ref{lem:dict}(V5)),
subject to the \emph{rise gate}: $g_{t-1}\ge 1$, or $t\ge 2$ and
$\ell_{t-1}<\ell^*$;
\item \emph{decorated $L$--wall, and $W^*$ with $\tau^*=L$:} the bare
multiplier $M_J$ \emph{(no $Q$)}, subject to a \emph{glue demand} on
its left neighbour: either $t\ge 2$, $g_{t-1}=0$, the head cluster of
wall $t-1$ is empty, and the pair gate holds ($d_{t-1}>d_t$, or
$d_{t-1}=d_t$ and $\ell_{t-1}<\ell_t$, for a decorated $L$--wall;
$d_{t-1}=0$ and $\ell_{t-1}<\ell^*$ for $W^*$ of type $L$); or, in
Case~B only, the last row of gap $t-1$ is $P$ with empty $P$--dress,
in which case the glued wall's $M_J$ occupies $P$'s dress slot
(Lemma~\ref{lem:Pblocks}).
\end{enumerate}
These are the wall operators in the absence of a glue demand from
wall $t+1$: if wall $t+1$ is glued to wall $t$ (case \textup{(5)}
read from the right), the head cluster of wall $t$ is forced
empty, so its factor $M_J$ (inside $\LL$ in \textup{(1)} and
\textup{(4)}, bare in \textup{(5)}) is suppressed, and the composite
operator of the resulting maximal chain is assembled in
Lemma~\ref{lem:chains}.
Moreover, whenever the gates in \textup{(3)} and \textup{(4)} are
tested with $g_{t-1}=0$ and wall $t-1$ of type $L$, the comparison
$\ell_{t-1}<\ell_t$ (resp.\ $\ell_{t-1}<\ell^*$) holds automatically;
consequently the factor $J$ of wall $t-1$ never splits: either every
term of $J$ is admissible, or none is.  (This non--splitting
assertion concerns the decoration gates \textup{(3)} and
\textup{(4)}; a glue demand does not split a factor $J$ either, but
for the different reason that it removes the factor wholesale: the suppression just described.)
\end{lemma}

\begin{proof}
(1) An undecorated $L$--wall attacks every later diagonal--$0$
active (decorated or not, including its own head cluster), giving the substitution $Q$ applied to everything on its right (in
Case~B it does not attack the diagonal--one row $P$, whose content
is carried by the transparent markers of Section~\ref{sec:caseB});
composing with the cluster's
content $M_J$ (Lemma~\ref{lem:cluster}) and commuting $Q$ past the
multiplier yields $Q\,M_J=J^{(q)}Q=\LL$.  By Lemma~\ref{lem:dict} it
receives only from an earlier $P$ (the scalar $q^{\Lambda}$, charged
to $P$ in Lemma~\ref{lem:Pblocks}) and from earlier outside rows
(inside $D(\Xi)$).

(2) A high wall attacks no later active (its label exceeds both
active labels, killing primary attacks, and a secondary attack from
diagonal $0$ has no target on diagonal $-1$); by
Lemma~\ref{lem:cluster} it hosts no cluster.

(3) A decorated $H$--wall must be a valley, i.e.\ (V4) must hold for
its pred.  If $g_{t-1}\ge 1$, its pred is the last row of gap $t-1$: a
diagonal--$0$ active (compound steps and clusters end in actives),
validating by the label inequality $r,r+1<\ell_t$, or, in Case~B,
possibly $P$ with empty dress, validating by its drop.  If $g_{t-1}=0$, the pred is
the last cluster letter of wall $t-1$ (an active, validating) when that cluster is nonempty, and wall $t-1$ itself otherwise,
validating iff $d_{t-1}=1$ (drop) or $\ell_{t-1}<\ell_t$.  When wall
$t-1$ is of type $L$, the comparison
$\ell_{t-1}<r<r+1<\ell_t$ holds automatically, so the cluster and
no--cluster terms of wall $t-1$'s factor $J$ are simultaneously
admissible, which is the final claim; when wall $t-1$ is high there
is no cluster and the comparison decides; $t=1$ with $g_0=0$ leaves
no valid pred.  Identical reasoning, with the rise condition (V1)
in place of (V4), and using that every active label is $<\ell^*$
when $\tau^*=H$, gives (4); note that $W^*$ of type $H$ does
attack every later active (a secondary attack, since $\ell^*>r+1$)
and does host a cluster (its successor is a drop), whence the
operator $\LL$ again.

(5) By (V4) a decorated $L$--wall admits \emph{no} active pred (every
active label exceeds $\ell_t$, and an active pred is a flat pred),
and by (V1) the same holds for $W^*$ of type $L$ (the rise needs a
pred label $<\ell^*<r$).  Hence its pred must be wall $t-1$ itself (so $t\ge 2$)
--- forcing $g_{t-1}=0$ \emph{and} the cluster of wall $t-1$ empty,
with the stated pair gate --- or, in Case~B, the row $P$ (a drop
pred, always valid, with no condition on labels); in the latter case
$P$ must be the last row of gap $t-1$, and $P$'s own dress cluster
must be empty, the glued wall's cluster occupying the unique dress
slot after $P$.  (For $t=1$ only the $P$--alternative is available;
failing it, the refined tuple is empty.)  The glued wall emits
nothing: a decorated $L$--wall
is screened, and $W^*$ of type $L$ attacks no later active since
$\ell^*<r$.  Its operator is therefore the bare cluster multiplier
$M_J$.
\end{proof}

\begin{lemma}[chains anchor]\label{lem:chains}
Call a maximal run of consecutive glued walls (type \textup{(5)}
of Lemma~\ref{lem:walls}) a \emph{glue chain}.  This lemma treats only
chains whose final demand is not satisfied by $P$; the $P$--anchored
case is part of Lemma~\ref{lem:Pblocks}.  Read leftwards, such a chain's
final demand lands on a wall of type \textup{(1)} or \textup{(4)} (the
\emph{anchor}), or the refined tuple is empty.  A chain with anchor of
type \textup{(1)} or \textup{(4)} contributes the single operator
$\LL$: the anchor supplies $Q$, every glued wall except the last has its
$M_J$ suppressed (empty cluster, forced by the next glue demand), and
the last glued wall supplies the unique surviving $M_J$, so the
composite is $Q\circ M_J=\LL$.  In particular, in any flat or Case~A
refined class, every surviving factor $J$ in the word is eventually
multiplied by a $Q$ on its left, and the operator word is, gates
permitting, of the form
\begin{equation}\label{eq:flatword}
q^{\,\mathrm{const}}\;\;
T^{g_0}\,\mathcal{W}_1\,T^{g_1}\,\mathcal{W}_2\cdots
\mathcal{W}_s\,T^{g_s}\,\onev,
\qquad \mathcal{W}_t\in\{\LL,\ \mathrm{Id}\},
\end{equation}
where $\mathcal{W}_t=\LL$ exactly when wall $t$ anchors a (possibly
singleton) chain (i.e.\ is an undecorated $L$--wall or $W^*$ with
$\tau^*=H$), and a refined tuple violating a gate contributes $0$.
The constant is explicit and color--blind:
$\mathrm{const}=\sum_t h_tg_t$, one factor $q^{h_t}$ per undecorated
active of gap $t$.  (After the normalization of
Proposition~\ref{prop:flatA}, pulling each $\mathcal{W}_t=\LL$ to the
left across $T^{g_0+\cdots+g_{t-1}}$ by~\eqref{eq:Lpull} contributes
the further explicit exponent, so altogether
$c(\Xi,g)=\sum_t h_tg_t-\sum_{t:\,\mathcal{W}_t=\LL}
(g_0+\cdots+g_{t-1})$ there; in Example~\ref{ex:sharp} this gives
$c=g_0$.)
\end{lemma}

\begin{proof}
A glue demand not satisfied by $P$ requires its left neighbour to be a
wall, with the pair gate.  If that neighbour is itself a glued wall, the
demand propagates (the pair gate between consecutive glued walls is the
stated label comparison), and the chain must terminate leftwards at a
wall that is not glued, i.e.\ of type \textup{(1)--(4)}, or at the start
of the word (empty tuple).  A diagonal--zero high--wall terminus, of
type \textup{(2)} or \textup{(3)}, fails the pair gate
($\ell_{H}>\ell_{L}$, and a diagonal--$0$ wall offers no drop); $W^*$
with $\tau^*=H$, on diagonal $1$, anchors by its drop.  Hence the
anchor is of type \textup{(1)} or \textup{(4)}.  Within the chain, each
glue demand forces the empty cluster of the wall to its left,
so only the rightmost glued wall retains its $M_J$ (this matches
the combinatorics, since the consecutive walls are adjacent rows and
a head cluster can only follow the last of them), and the anchor's
own cluster is forced empty by the first demand, so the anchor
contributes its bare $Q$.  The composite operator of the chain, read
in word order \emph{anchor, glued walls, final cluster}, is
$Q\circ M_J=\LL$.  Assembling Lemma~\ref{lem:gapsteps} for the gaps
and Lemma~\ref{lem:walls} for the walls
yields~\eqref{eq:flatword}; the scalars $q^{h_t}$ attach one factor
per undecorated active of gap $t$, giving $q^{\sum_t h_tg_t}$,
manifestly independent of the colors and of the decoration pattern
within the gaps.
\end{proof}

\begin{remark}[local transition table]\label{rem:table}
For reference, the complete list of local blocks and their operators, including the Case~B blocks of
Section~\ref{sec:caseB}, is:
\begin{center}
\small
\begin{tabular}{lll}
\toprule
block (read left to right) & operator & gate / proviso\\
\midrule
compound step in gap $t$ (Lem.~\ref{lem:gapsteps}) & $q^{h_t}\,T$ & ---\\
undec.\ $L$--wall $+$ cluster (Lem.~\ref{lem:walls}(1)) & $\LL$ & ---\\
undec.\ $H$--wall (Lem.~\ref{lem:walls}(2)) & $\mathrm{Id}$ & ---\\
dec.\ $H$--wall (Lem.~\ref{lem:walls}(3)) & $\mathrm{Id}$ & gate (3)\\
$W^*$, $\tau^*{=}H$, $+$ cluster (Lem.~\ref{lem:walls}(4)) & $\LL$ & rise gate (4)\\
dec.\ $L$--wall; $W^*$, $\tau^*{=}L$ (Lem.~\ref{lem:walls}(5)) & $M_J$ & glue demand (5)\\
$P$--block (a) (Lem.~\ref{lem:Pblocks}) & $q^{h_t+\Lambda_t}\,\Mu\,\hat p_2\,\LL$ & ---\\
$P$--block (b)$+$(c) (Lem.~\ref{lem:Pblocks}) & $q^{\Lambda_t}(\hat p_1\LL+\hat p_2\ddb\LL)$ & host wall's gate/glue\\
\bottomrule
\end{tabular}
\end{center}
A glue demand from the wall to the right suppresses the $J$--factor
of the block to its left (its head cluster, or $P$'s dress); a
maximal chain of glued walls contributes the single operator $\LL$
at its anchor (Lemma~\ref{lem:chains}), or leaves the $P$--block
composites unchanged when the anchor is $P$
(Lemma~\ref{lem:Pblocks}).  Every admissible word is a concatenation
of these blocks, and each block's operator and gate depend only on
$\Xi$, the tuple $g$, and the placement of $P$, never on the
colors or the decoration pattern inside the gaps.
\end{remark}

\begin{proposition}[assembly: flat case and Case A]\label{prop:flatA}
Let $\sortw(m)\in\{0^n,0^{n-1}1\}$ and let the class $(m,k,\Xi)$ be
flat or of Case~A.  For every refined tuple $g$, either
$\Phi^{(r)}_{m,k,\Xi,g}=0$, or
\[
\Phi^{(r)}_{m,k,\Xi,g}(u,v;q)
\;=\;
q^{\,D(\Xi)-k}\;q^{\,c}\;
\Bigl[\,\text{$(u,v)$--degree }N,\ z\text{--degree }\kappa\,\Bigr]
\;\LL^{\,k_Q}\,T^{\,|g|}\,\onev ,
\]
where $|g|=\sum_tg_t$, $\kappa=N-|g|$,
$k_Q=\#\{\text{undecorated }L\text{--walls}\}+[\tau^*{=}H]$ (with
$[\tau^*{=}H]:=0$ in the flat case, where $W^*$ is absent),
$c=c(\Xi,g)\in\mathbb{Z}$ is the explicit color--blind exponent of
Lemma~\ref{lem:chains}, and the bracket denotes
extraction of the indicated bigraded component.  In particular
$\Phi^{(r)}_{m,k,\Xi,g}\in\sym$.
\end{proposition}

\begin{proof}
By Lemmas~\ref{lem:dict} and~\ref{lem:chains} the class generating
function is $q^{D(\Xi)-k}$ times the indicated extraction of the
word~\eqref{eq:flatword}: the extraction imposes the class size.  The
variable $z$ marks only decorated active rows; outside decorations are
fixed by the scaffold flags and contribute only through the scalar
$q^{D(\Xi)-k}$.  Each active row contributes exactly one factor $u$ or
$v$, while the $|g|$ compound steps account for the undecorated active
rows, so total $(u,v)$--degree $N$ and $z$--degree $\kappa=N-|g|$
force precisely the decorated active rows of the refined class.  Pulling
every $\LL$ to the left
with~\eqref{eq:Lpull} turns the word into
$q^{-e}\,\LL^{k_Q}T^{|g|}\onev$ for a color--blind integer $e$: each
pull of one $\LL$ across $T^{a}$ contributes $q^{-a}$, a function of
$g$ and the wall positions only.  By Theorem~\ref{thm:identityI},
$T^{|g|}\onev\in\sym$; by Lemma~\ref{lem:comm}(7),
$\LL^{k_Q}T^{|g|}\onev\in\sym$; and extraction of a bigraded
component preserves $\sym$.
\end{proof}

\begin{corollary}[empty scaffold]\label{cor:empty}
For $\Xi=\varnothing$, $m=0^{n}$, and $0\le k\le n$,
\[
\Phi^{(r)}_{0^n,k,\varnothing}(u,v;q)
=q^{-k}\,
\Bigl[\,\text{$(u,v)$--degree }n,\ z\text{--degree }k\,\Bigr]\,
T^{\,n-k}\onev,
\]
which is symmetric; for $k>n$ the class is empty and the generating
function vanishes.  This is the wall--free area--zero exchange
kernel.
\end{corollary}

\begin{proof}
With no walls there is a single gap, one refined tuple
$g=(n-k)$, no gates, and $D(\varnothing)=0$; apply
Proposition~\ref{prop:flatA}.
\end{proof}

\section{Case B: the diagonal--one active row}\label{sec:caseB}

In Case B all walls have diagonal $0$ and one active row $P$ sits on
diagonal $1$.  Lemma~\ref{lem:dict} gives: $P$ is never decorated;
$P$ emits a secondary attack to \emph{every} later low wall, decorated or not, contributing the scalar $q^{\Lambda_t}$ if $P$
lies in gap $t$, for either color of $P$; if $P$ has color $2$ it
emits to every later active of color $1$, the substitution
$\sigma_u$; and $P$ receives nothing: earlier rows attack $P$ in
no case.  Consequently $P$'s own content variable must be
\emph{shielded} from every operator in the word.  Formally, we
enlarge the ring to $\widehat R:=R[\hat p_1,\hat p_2]$ and extend
$\sigma_u,\sigma_v$ (hence $\dd,\ddb,Q,T,\LL$ and every
multiplier $M_F$, $F\in R$) to $\widehat R$ by letting them fix
$\hat p_1,\hat p_2$ and act on the $R$--coefficients; the markers
are thus \emph{transparent} to all operators.  The row $P$
contributes the marker $\hat p_1$ (color $1$) or $\hat p_2$ (color
$2$), and only in the final polynomial do we apply the
$\mathbb{Z}[q,q^{-1}]$--algebra substitution $\hat p_1\mapsto u$,
$\hat p_2\mapsto v$.  By (V2) the pred
of $P$ is an undecorated active of color $1$, or the decorated head
$1^*$ of an $L$--type wall's cluster (in which case $P$ has color
$2$, forced, since $(r,r)$ is not an ascent) or an $L$--type wall
directly.  These are the placements (a), (b), (c) below; (b) and
(c) require the wall $t$, hence $t\ge1$, and in gap $0$ only (a)
occurs, since the word begins with an undecorated diagonal--$0$
active and $P$ is never its first row.  The row
following $P$ sees a drop, so $P$ enables its own \emph{dress}
cluster with content $J$: Lemma~\ref{lem:cluster} applies verbatim
with $P$ in place of the wall.  The proofs of this section invoke
Lemmas~\ref{lem:walls} and~\ref{lem:chains} only through clauses in
which no glue demand is satisfied by $P$.  When a glue demand is
satisfied by $P$, the needed structure is derived directly in the glued
case of Lemma~\ref{lem:Pblocks} below from Lemmas~\ref{lem:dict} and~\ref{lem:cluster}.

Throughout, the \emph{$P$--block} comprises the rows removed in
Lemma~\ref{lem:gapsteps} (its \emph{active part}: $P$, its dress,
and $P$'s active pred when present) together with the operator slots
whose validity is anchored at $P$.  Thus placements (b) and (c)
include the host wall's slot; when that host is decorated, its
leftward contracted glue chain of Lemma~\ref{lem:chains} is included
as well.  In every placement, including (a), a maximal rightward glue
chain whose first demand is satisfied by $P$ is also contracted into
the block, as described in the glued clause below.

\begin{lemma}[$P$--blocks]\label{lem:Pblocks}
Fix a gap $t$; write $X$ for the operator word of everything to the
right of the $P$--block, and $h=h_t$, $\Lambda=\Lambda_t$.
\begin{enumerate}
\item[(a)] \emph{$P$ after an undecorated $1$ inside gap $t$.}  The
compound (undecorated $1$, then $P$ of forced color $2$, then $P$'s
dress) contributes the operator
$q^{\,h+\Lambda}\;\Mu\,\hat p_2\;\LL$, inserted between the compound
steps of the gap; the factor $q^{h}$ is carried by the undecorated
$1$, while $P$ carries the separate low--wall scalar $q^{\Lambda}$ and
no $q^h$--factor.
\item[(b)$+$(c)] \emph{$P$ at the start of gap $t$, after wall $t$ of
type $L$ (either flag; $t\ge1$).}  Summing the placements --- $P$ directly
after the wall, both colors (c), and $P$ after the head $1^*$, color
$2$ forced (b) --- composed with wall $t$'s own operator and dress
slot, yields at wall $t$'s position the operator
\[
q^{\Lambda}\,\bigl(\hat p_1\,\LL\;+\;\hat p_2\,\ddb\,\LL\bigr)
\;=\;q^{\Lambda}\,\ddb\,(\hat p_1+\hat p_2\,\ddb)\,\dd
\]
when wall $t$ is an undecorated $L$--wall; when wall $t$ is a
decorated $L$--wall the same composite arises once the glue chain of
Lemma~\ref{lem:chains} reaches its anchor's $Q$, and wall $t$'s own
glue demand applies unchanged toward wall $t-1$.
\item[(glued)] If a decorated $L$--wall immediately to the right of $P$
has its predecessor demand satisfied by $P$, then $P$ is the last row
of the preceding
gap and in \textup{(a)--(c)} the factor $J$ of $P$'s dress is
replaced by $1$ while that wall's $M_J$ takes the slot, leaving each
composite operator unchanged.  The same holds for a maximal glue
chain anchored at $P$: the dress slot and the $M_J$ of every glued
wall except the last are suppressed, the last glued wall supplies the
unique surviving factor $J$, and each composite operator is again
unchanged.
\end{enumerate}
\end{lemma}

\begin{proof}
(a) The pred $1$ contributes $u$ (attackable in place: earlier
$L$--walls and $\LL$'s act on this $u$), together with its attacks
on later high walls ($q^{h}$), and on later $2$'s including the
$2^*$ of $P$'s dress ($\sigma_v$, the substitution supplying the $q$
on the dress's $zv$).  $P$ contributes $\hat p_2$, the scalar
$q^{\Lambda}$, and $\sigma_u$ on everything later, including the
$zu$ of its own dress; the dress contributes $M_J$.  Composing,
\[
q^{h+\Lambda}\,\Mu\,\hat p_2\,\sigma_u\sigma_v M_J
=q^{h+\Lambda}\,\Mu\,\hat p_2\,Q M_J
=q^{h+\Lambda}\,\Mu\,\hat p_2\,\LL.
\]
(The color of $P$ is forced to $2$ by (V2): its pred has label $r$.)

(b)$+$(c) Let first wall $t$ be undecorated, with operator
$Q\circ(\text{dress slot})$.  Placement (c) with $P$ of color $1$:
the slot holds $P$, i.e.\ $q^{\Lambda}\hat p_1$, then $P$'s dress
$M_J$; composing with the wall's $Q$:
\[
Q\bigl(q^{\Lambda}\hat p_1\,M_JX\bigr)
=q^{\Lambda}\hat p_1\,J^{(q)}QX
=q^{\Lambda}\hat p_1\,\LL X.
\]
Placement (c) with $P$ of color $2$ adds $\sigma_u$:
\[
Q\bigl(q^{\Lambda}\hat p_2\,\sigma_u(M_JX)\bigr)
=q^{\Lambda}\hat p_2\,(1+q^{2}zu)(1+qzv)\,\sigma_u^{2}\sigma_vX.
\]
Placement (b): the cluster $\{1^*\}$ of wall $t$ contributes $zu$,
receiving $q$ from the wall's $Q$, then $P$ of color $2$ as before:
\[
Q\bigl(zu\;q^{\Lambda}\hat p_2\,\sigma_u(M_JX)\bigr)
=q^{\Lambda+1}zu\,\hat p_2\,(1+q^{2}zu)(1+qzv)\,\sigma_u^{2}\sigma_vX.
\]
Adding (b) to (c, color $2$):
\[
q^{\Lambda}\hat p_2\,(1+qzu)(1+q^{2}zu)(1+qzv)\,
\sigma_u^{2}\sigma_v X
\;=\;q^{\Lambda}\hat p_2\,\ddb^{\,2}\dd\,X ,
\]
since $\ddb^{2}=(1+qzu)(1+q^{2}zu)\sigma_u^{2}$; and (c, color $1$)
is $q^{\Lambda}\hat p_1\,\ddb\dd\,X$.  The displayed identity in the
statement follows from $\LL=\dd\ddb=\ddb\dd$ and
$\ddb\LL=\ddb^{2}\dd$.  For a decorated $L$--wall $t$, the wall's
operator is the bare slot $M_J$, now filled by the $P$--material;
the composite at wall $t$ is
$q^{\Lambda}M_J(\hat p_1+\hat p_2\ddb)$ before any $Q$.  If wall
$t$'s leftward glue chain fails to anchor, the refined tuple is
empty (Lemma~\ref{lem:chains}) and the placement contributes $0$;
otherwise applying
the eventual anchor's $Q$ (Lemma~\ref{lem:chains}) converts it to
\[
q^{\Lambda}\,J^{(q)}\,
\bigl(\hat p_1+\hat p_2(1+q^{2}zu)\sigma_u\bigr)\,Q
=q^{\Lambda}\bigl(\hat p_1\,\LL+\hat p_2\,\ddb\,\LL\bigr),
\]
the same composite; intermediate glued walls are transparent.  (We
used $Q\,\sigma_u(1+zu)(1+zv)
=(1+q^{2}zu)(1+qzv)\,\sigma_u^{2}\sigma_v$ and
$J^{(q)}(1+q^{2}zu)\,\sigma_u Q=\ddb^{2}\dd$.)

(glued) Suppose first that a glued wall immediately follows $P$.
By Lemma~\ref{lem:dict}\textup{(V4)}, a decorated $L$--wall cannot
use a flat active predecessor: every diagonal--$0$ active has label
$r$ or $r+1$, larger than the wall label.  Hence a decorated
$L$--wall whose demand is satisfied by an active predecessor must be
preceded by the diagonal--one row $P$, where the diagonal drop makes the
demand valid.  Thus $P$ is last in
gap $t$; its own dress is empty, and wall $t+1$'s $M_J$ occupies the
unique slot after $P$.  Since that slot holds exactly one factor $J$
in either reading (a second cluster cannot follow a $2^*$, nor a
$1^*$ follow the first cluster's $2^*$, by Lemma~\ref{lem:cluster}),
the operator composites of \textup{(a)--(c)} are literally unchanged,
the pair gate of the glue being replaced by the always--valid drop
pred $P$.  If walls $t+1,\dots,t+j$ form a maximal glue chain
anchored at $P$, the demand of wall $t+i+1$ forces the cluster of
wall $t+i$ empty by the same predecessor argument and
Lemma~\ref{lem:cluster}: the glued walls are consecutive rows, so a
head cluster can follow only the last of them; every suppressed wall
is a screened decorated $L$--wall, emitting nothing; and the unique
surviving factor $J$ again occupies the
single slot after the last row of the block.  The composite operators
of \textup{(a)--(c)} are therefore unchanged for chains of any
length.
\end{proof}

The (a)--pieces are handled by Proposition~\ref{prop:flatA} verbatim.

\begin{lemma}\label{lem:apieces}
Each \textup{(a)}--placement contributes to the refined class
generating function a term of the form
$q^{c}\,u\,\hat p_2\;\LL^{k'}T^{m'}\onev$, for some $c\in\mathbb Z$
and $k',m'\in\mathbb Z_{\ge0}$ determined by the class and the
placement, extracted in the forced bidegree, with $c$ color--blind; after $\hat p_2\mapsto v$ this is
$q^{c}\,uv\cdot\LL^{k'}T^{m'}\onev\in uv\cdot\sym\subset\sym$.
\end{lemma}

\begin{proof}
After deleting the $P$--block the remaining wall/gap word consists
of diagonal--$0$ letters only and is a flat word of the
form~\eqref{eq:flatword}, any $P$--anchored glue chain being
absorbed into the block (Lemma~\ref{lem:Pblocks}); insert the block
of Lemma~\ref{lem:Pblocks}(a) at its position in this word.  By Lemma~\ref{lem:comm}(6) the operator
$\Mu$ commutes with $T$, and $\LL\Mu=q\Mu\LL$; since $\hat p_2$ is
transparent, both prefactors migrate to the far left at the cost of
color--blind scalars, and pulling the $\LL$'s left
by~\eqref{eq:Lpull} gives the normal form.  The factor $uv$ is
$\swap$--invariant, $\LL^{k'}T^{m'}\onev\in\sym$ by
Theorem~\ref{thm:identityI} and Lemma~\ref{lem:comm}(7), and
extraction preserves $\sym$.
\end{proof}

For the (b)$+$(c)--pieces, the new symmetric family is the
following.

\begin{lemma}[Lemma $\BB$]\label{lem:B}
For all integers $g,M\ge 0$,
\[
\BB(g,M)\;:=\;u\,T^{\,g+M}\onev\;+\;v\,T^{\,g}\,\ddb\,T^{\,M}\onev
\;\in\;\sym.
\]
\end{lemma}

\begin{proof}
\emph{Base $g=0$.}
$\BB(0,M)=u\,T^M\onev+v\,\ddb\,T^M\onev=\bar T\,T^{M}\onev
=T^{M+1}\onev\in\sym$, by Corollary~\ref{cor:TbarT} and
Theorem~\ref{thm:identityI}.

\emph{Recursion.}  We claim, as an identity of polynomials,
\begin{equation}\label{eq:Brec}
\BB(g+1,M)\;=\;q^{-1}\,\BB(g,M+1)
\;+\;(1-q^{-1})\Bigl[(u+v)\,T^{\,g+M+1}\onev-uv\,T^{\,g+M}\onev
\Bigr].
\end{equation}
Granting~\eqref{eq:Brec}, induction on $g$ finishes the proof: the
bracket lies in $\sym$ by Theorem~\ref{thm:identityI}, since $u+v$
and $uv$ are symmetric multipliers, and $\BB(g,M+1)\in\sym$ by the
inductive hypothesis, which is quantified over all $M$.

To prove~\eqref{eq:Brec}, start from
$T\ddb=\ddb\,T^{\langle 1/q\rangle}$, Lemma~\ref{lem:comm}(5):
\[
T^{g+1}\ddb\,T^M\onev
= T^{g}\,\ddb\;T^{\langle 1/q\rangle}T^{M}\onev.
\]
Now $T^{\langle 1/q\rangle}=q^{-1}u\dd+v$ and
$u\dd\,T^M\onev=(T-v)\,T^M\onev=T^{M+1}\onev-v\,T^M\onev$, so
\[
T^{\langle 1/q\rangle}T^M\onev
= q^{-1}\,T^{M+1}\onev+(1-q^{-1})\,v\,T^{M}\onev,
\]
whence
\[
v\,T^{g+1}\ddb\,T^M\onev
= q^{-1}\,v\,T^{g}\ddb\,T^{M+1}\onev
+(1-q^{-1})\,v\,T^{g}\,\ddb\bigl(v\,T^{M}\onev\bigr).
\]
For the last term, Lemma~\ref{lem:comm} gives
$\ddb(vG)=v\,\ddb G$ by part (2),
$T^{g}(vK)=v\,(T^{\langle q\rangle})^{g}K$ by part (6), and
$(T^{\langle q\rangle})^{g}\ddb=\ddb\,T^{g}$ by iterating part
(5).  Therefore
\[
v\,T^{g}\,\ddb\bigl(v\,T^M\onev\bigr)
= v^{2}\,\ddb\,T^{\,g+M}\onev
= v\,(\bar T-u)\,T^{\,g+M}\onev
= v\,T^{\,g+M+1}\onev-uv\,T^{\,g+M}\onev,
\]
using $v\ddb=\bar T-u$ and Corollary~\ref{cor:TbarT}.  Adding
$u\,T^{g+1+M}\onev$ and comparing with
$q^{-1}\BB(g,M+1)=q^{-1}u\,T^{g+M+1}\onev
+q^{-1}v\,T^g\ddb\,T^{M+1}\onev$
yields exactly~\eqref{eq:Brec}.
\end{proof}

\begin{proposition}[assembly: Case B]\label{prop:caseB}
Let $m=0^{n-1}1$ and let the class $(m,k,\Xi)$ be of Case B.  For
every refined tuple $g$, the generating function
$\Phi^{(r)}_{m,k,\Xi,g}$ is the extraction, in total $(u,v)$--degree
$N$ and $z$--degree $\kappa=N-|g|$ after the substitution
$\hat p_1\mapsto u$, $\hat p_2\mapsto v$ (so that $P$ counts toward
the $(u,v)$--degree and not the $z$--degree), of a finite sum of
terms of the two forms
\[
q^{c}\;uv\cdot\LL^{k'}T^{m'}\onev
\qquad\text{and}\qquad
q^{c}\;\LL^{k'}\,\BB(G,M),
\]
for some $c\in\mathbb Z$ and $k',m',G,M\in\mathbb Z_{\ge0}$
determined by the class and the placement, with color--blind
exponents; each $q^{c}$ includes the global
factor $q^{\,D(\Xi)-k}$ of Lemma~\ref{lem:dict}, exactly as in
Proposition~\ref{prop:flatA}.  In particular
$\Phi^{(r)}_{m,k,\Xi,g}\in\sym$; more precisely, each
\textup{(a)}--placement, and each \textup{(b)}$+$\textup{(c)} block (the sum over the admissible head/color alternatives of $P$ at a
fixed $L$--type wall) contributes a term in $\sym$.
\end{proposition}

\begin{proof}
Sum over the position of $P$ (the gap $t$ and the placement type), which is not part of the refined data.  By Lemma~\ref{lem:dict}\textup{(V2)}, the predecessor of $P$ is exactly one of: an undecorated active $1$, the decorated head $1^*$ of an $L$--type wall's head cluster, or an $L$--type wall directly; these mutually exclusive alternatives are precisely placements \textup{(a)}, \textup{(b)}, and \textup{(c)}, and they exhaust the nonempty placements of $P$.  The global factor
$q^{\,D(\Xi)-k}$ of Lemma~\ref{lem:dict} multiplies every
placement and is absorbed into the exponents $c$.  The
(a)--placements give
the first form by Lemma~\ref{lem:apieces}.  For a
(b)$+$(c)--block at an $L$--type wall $t$ (Lemma~\ref{lem:Pblocks}),
write the suffix to its right in the normal form
$X=q^{-c_x}\,\LL^{K}T^{M}\onev$, obtained as in the proof of
Proposition~\ref{prop:flatA}: after absorbing any rightward glue chain
whose first demand is satisfied by $P$ into the $P$--block as in the
glued clause of Lemma~\ref{lem:Pblocks}, the suffix is a flat
$P$--free word with all gates and glue demands already decided.  Then,
using $\ddb\LL=\LL\ddb$,
\[
q^{\Lambda}\bigl(\hat p_1\LL+\hat p_2\ddb\LL\bigr)X
= q^{\Lambda-c_x}\bigl(\hat p_1\,\LL^{K+1}T^{M}\onev
+\hat p_2\,\LL^{K+1}\ddb\,T^{M}\onev\bigr).
\]
When wall $t$ is decorated, first contract its maximal leftward
glue chain (Lemma~\ref{lem:Pblocks}): if the chain fails to anchor,
the placement contributes $0$; otherwise the anchor's $Q$, the
suppressed intermediate walls, and the material at wall $t$ compose
into the block operator
$q^{\Lambda}\bigl(\hat p_1\,\LL+\hat p_2\,\ddb\,\LL\bigr)$ of
Lemma~\ref{lem:Pblocks}, and ``the block'' means this contraction.
The prefix to the left of the block is a flat word
$T^{a_0}\mathcal{W}_1T^{a_1}\cdots$ as in~\eqref{eq:flatword}; since
$\hat p_1,\hat p_2$ are transparent, each prefix factor acts on the
two summands separately.  Every $T^a$ sends
$\LL^{K+1}T^{j}\onev\mapsto q^{-a(K+1)}\LL^{K+1}T^{a+j}\onev$ and
$\LL^{K+1}T^{j}\ddb\,T^M\onev\mapsto
q^{-a(K+1)}\LL^{K+1}T^{a+j}\ddb\,T^M\onev$ by~\eqref{eq:Lpull},
the $\ddb$ staying at distance $M$ from the right; every prefix
$\LL$ raises $K$ by one.  Hence the block's total contribution,
after the final substitution, is
\[
q^{c}\bigl(u\,\LL^{k'}T^{\,G+M}\onev
+v\,\LL^{k'}T^{\,G}\ddb\,T^{M}\onev\bigr),
\]
where $G$ collects the gap steps between the block and the start of
the word and $c$ is color--blind.  Finally, by
Lemma~\ref{lem:comm}(4),
$\LL^{k'}(uA+vB)=q^{k'}\bigl(u\,\LL^{k'}A+v\,\LL^{k'}B\bigr)$ for
arbitrary $A,B$, so the display equals
$q^{c-k'}\,\LL^{k'}\BB(G,M)$, of the second form; it lies in $\sym$
by Lemma~\ref{lem:B} and Lemma~\ref{lem:comm}(7).  The glued
configurations of Lemma~\ref{lem:Pblocks} produce the same
composites.  Gates may annihilate a placement, and a glue demand may
suppress a $J$--factor wholesale (Lemmas~\ref{lem:walls}
and~\ref{lem:chains}).  Outside the explicit $P$--block decomposition
of Lemma~\ref{lem:Pblocks}, no admissibility condition splits a
surviving factor $J$ term by term (Lemma~\ref{lem:walls}); inside the
$P$--block, the only split is the grouped \textup{(b)}$+$\textup{(c)}
contribution just proved symmetric.  Hence the refined class generating
function is the sum of the listed symmetric terms.
\end{proof}

\section{Proofs of the main theorems}\label{sec:proofs}

\begin{theorem}[refined scaffold exchange in area at most one]
\label{thm:main}
Let $n\ge1$, let $m$ be a size--$n$ diagonal multiset with
$\sortw(m)\in\{0^{n},0^{n-1}1\}$ (the second case occurring only for
$n\ge2$), let $k\ge 0$, $r\ge 1$, let $\Xi$ be any scaffold of
length $s$, and let $g=(g_0,\dots,g_s)\in\mathbb{Z}^{s+1}$ be
interpreted as per--gap undecorated--active counts.  Then
\[
\Phi^{(r)}_{m,k,\Xi,g}(u,v;q)\;=\;\Phi^{(r)}_{m,k,\Xi,g}(v,u;q).
\]
\end{theorem}

\begin{proof}
Classes with incompatible data (Section~\ref{sec:dictionary}) and
refined tuples outside the convention of
Definition~\ref{def:refined} are empty, hence symmetric.
Propositions~\ref{prop:flatA} and~\ref{prop:caseB} cover the three
regimes (flat, Case A, Case B), which exhaust the nonempty
classes with $\sortw(m)\in\{0^n,0^{n-1}1\}$.
\end{proof}

\begin{remark}[per--diagonal refinement]\label{rem:perdiag}
The proof gives slightly more than the statement.  In Case~B one may
additionally fix the gap containing the unique diagonal--one active row
$P$, and symmetry persists: by Proposition~\ref{prop:caseB} every
\textup{(a)}--placement and every \textup{(b)}$+$\textup{(c)} block
contributes a symmetric term, and the placements of $P$ inside a
fixed gap partition into blocks of exactly these two kinds, so the
sum over the placements inside any one gap is symmetric.  In the notation of
Conjecture~\ref{conj:refined} below, the \emph{per--diagonal} refined
exchange holds for every $m$ of area at most one.
\end{remark}

\begin{theorem}\label{thm:conj}
The scaffold exchange conjecture
\textup{(}Conjecture~\ref{conj:exchange}\textup{)} holds for every diagonal multiset of area at most
one: for every $n\ge1$, $\sortw(m)\in\{0^{n},0^{n-1}1\}$, all
$k\ge0$ and $r\ge1$,
and every scaffold $\Xi$, the
class generating function $\Phi^{(r)}_{m,k,\Xi}(u,v;q)$ is symmetric
in $u,v$.
\end{theorem}

\begin{proof}
By Definition~\ref{def:refined},
$\Phi^{(r)}_{m,k,\Xi}=\sum_{g}\Phi^{(r)}_{m,k,\Xi,g}$, and each
summand is symmetric by Theorem~\ref{thm:main}.  (By
Remark~\ref{rem:D}, fixing the auxiliary datum $D$ in addition would
refine nothing, since $D=D(\Xi)$.)
\end{proof}

\begin{corollary}\label{cor:slices}
For all $n\ge1$ and $k\ge0$, the slices $\Vall_{n,k,0^{n}}$ and
(for $n\ge2$) $\Vall_{n,k,0^{n-1}1}$ are symmetric functions;
equivalently, the coefficients of $t^{0}$ and of $t^{1}$ in
$\Val_{n,k}(x;q,t)$ are symmetric functions.  (Slices with no
objects, e.g.\ $k$ too large, vanish; for $n=1$ the $t^{1}$
coefficient is $0$ by Remark~\ref{rem:lowarea}.)
\end{corollary}

\begin{proof}
Theorem~\ref{thm:conj} and Proposition~\ref{prop:reduction}, together
with Remark~\ref{rem:lowarea}: $0^n$ is the unique area--zero
multiset, for $n\ge2$ the multiset $0^{n-1}1$ is the unique
area--one multiset, and for $n=1$ the $t^{1}$ coefficient vanishes.
\end{proof}

\begin{remark}[relation to the $t=0$ specializations]\label{rem:QW}
At $t^{0}$, symmetry also follows from the theorem of Qiu and Wilson
\cite{QiuWilson}, who identified the $t=0$ specialization of the
(extended) valley side with an explicit symmetric function via
ordered multiset partitions, in the line of
\cite{HRS,Rhoades}; see also the $q=0$ specialization in the same
paper.  Corollary~\ref{cor:slices} at area zero is thus a new proof
by a structurally different route (it proceeds class by refined
class, never identifying the sum with a symmetric--function
expression), and the area--one statement is, to our knowledge, the
first symmetry result for any positive--area slice of the valley
side.  We emphasize what is \emph{not} claimed: we do not identify
$\Vall_{n,k,0^{n-1}1}$ with the corresponding slice of
$\Delta'_{e_{n-k-1}}e_n$, which would amount to the $t^{1}$ case of
the valley Delta conjecture itself.
\end{remark}

\section{Complements}\label{sec:complements}

\subsection{No letter--level intertwiner}\label{ssec:noR}

The classical route to exchange identities, as in \cite{HHL}, is a
local involution: an $R$--matrix at the level of single letters.  The
following observation shows that in the present calculus no such
operator exists, already in the flat undecorated case; the burden of
symmetry is genuinely carried by the closed series identities of
Theorem~\ref{thm:identityI} and Lemma~\ref{lem:B}.

\begin{proposition}\label{prop:noR}
Let $A_1:=\Mu\,\dd$ and $A_2:=\Mv$, the letter operators of the
model, so that $T=A_1+A_2$.  If an operator $\Psi$ on
$R=\mathbb{Z}[q,q^{-1}][u,v,z]$, linear over $\mathbb{Z}[q,q^{-1}]$,
satisfies the intertwining relations
\[
\Psi\,A_1=A_2\,\Psi,
\qquad
\Psi\,A_2=A_1\,\Psi,
\]
then $\Psi=0$.  The same holds for the undecorated specialization
$A_1=u\,\sigma_v$, $A_2=\Mv$ acting on
$\mathbb{Z}[q,q^{-1}][u,v]$.
\end{proposition}

\begin{proof}
By Lemma~\ref{lem:comm}(2),
$A_1A_2=\Mu\dd\Mv=q\,\Mu\Mv\dd=q\,A_2A_1$: the letter operators
satisfy the quantum--plane relation.  Hence
\[
A_2A_1\Psi \;=\; \Psi\,A_1A_2 \;=\; q\,\Psi\,A_2A_1 \;=\; q\,A_1A_2\,\Psi
\;=\; q^{2}\,A_2A_1\Psi ,
\]
so $(1-q^{2})\,A_2A_1\Psi=0$.  The ring $R$ is a free
$\mathbb{Z}[q,q^{-1}]$--module and $1-q^{2}\neq 0$, so
$A_2A_1\Psi=0$; and $A_2A_1=uv\,\dd$ is injective ($\sigma_v$ is
bijective, and $(1+qzv)$, $u$, $v$ are non--zero--divisors), whence
$\Psi=0$.  The specialization $z=0$ is identical with
$A_2A_1=uv\,\sigma_v$.
\end{proof}

In other words, an exchange involution at the letter level would
have to conjugate the quantum--plane parameter $q$ into $q^{-1}$,
and only the zero operator intertwines the two orderings.  It is
instructive to contrast this with what \emph{is} true: by
Lemma~\ref{lem:comm}(1), $\swap\,T=\bar T\,\swap$, and
Theorem~\ref{thm:identityI} is equivalent to the statement
$T^{m}\onev=\bar T^{m}\onev$ for all $m$, an identity that holds
on the cyclic vector $\onev$, emphatically not as an identity of
operators.  Together with the relations of Lemma~\ref{lem:comm},
this identifies, for the diagonal set $\{0,1\}$, the precise way in
which the quantum--torus transfer formalism extends to decorated
rows: the decorated letters appear not as new generators
but through the dressed multipliers $(1+q^{\bullet}zu)$,
$(1+q^{\bullet}zv)$ inside $\dd,\ddb$; the walls and the $P$--block
act by the words $\LL=\dd\ddb$, $\ddb\LL$, $\Mu\LL$ in the algebra
$\langle \Mu,\Mv,\dd,\ddb\rangle$, a two--variable quantum torus
in which the element $\LL$ is central up to the scaling
$\LL\Mu=q\Mu\LL$, $\LL\Mv=q\Mv\LL$; and the exchange itself is
carried by the two scalar identities, with
Proposition~\ref{prop:noR} showing that no operator--local witness
can exist.

\subsection{The refined conjecture in general, and area
two}\label{ssec:refined}

Definition~\ref{def:refined} extends verbatim to an arbitrary
diagonal multiset $m$, where active rows may occupy several
diagonals: for an object $O$ and a gap $t$, let
$g_{t,\delta}(O)$ be the number of undecorated active rows on
diagonal $\delta$ in gap $t$.  Theorem~\ref{thm:main},
Remark~\ref{rem:perdiag}, and the computations reported below
support:

\begin{conjecture}[refined scaffold exchange]\label{conj:refined}
For every $n\ge1$, every diagonal multiset $m$ realizable by an
area word of size $n$, every $k\ge0$ and $r\ge1$, every scaffold
$\Xi$, and every matrix $g=(g_{t,\delta})$ of per--gap,
per--diagonal undecorated--active counts (incompatible data and
inadmissible $g$ giving $0$, as in Section~\ref{sec:dictionary}),
the refined class generating function
$\Phi^{(r)}_{m,k,\Xi,g}(u,v;q)$ is symmetric in $u$ and $v$.
\end{conjecture}

Conjecture~\ref{conj:refined} implies Conjecture~\ref{conj:exchange}
by summing over $g$, and strictly refines it; by
Theorem~\ref{thm:main} and Remark~\ref{rem:perdiag} it holds for
$\area(m)\le 1$.  The first genuinely new case is area two: for $n\ge3$,
$\sortw(m)=0^{n-2}1^{2}$ is the unique multiset (a row on diagonal
$2$ would need a predecessor on diagonal at least $1$, so
$0^{n-1}2$ does not occur; for $n<3$ no area--two multiset is
realizable).  We verified
Conjecture~\ref{conj:refined} at area two by brute--force
enumeration, exhaustively over the following finite ranges:

\begin{center}
\begin{tabular}{lcccc}
\toprule
& $n=4,\ M=4$ & $n=5,\ M=4$ & $n=5,\ M=5$ & $n=6,\ M=4$\\
\midrule
refined classes $(g_{t,\delta})$ & $462$ & $3116$ & $18895$ & $17670$\\
asymmetric & $0$ & $0$ & $0$ & $0$\\
\bottomrule
\end{tabular}
\end{center}

\noindent
Here $n$ is the path size, $M$ bounds the label alphabet, and all
$k$, $r$, $\Xi$ within range are included; in total $40{,}143$
class verifications, each symmetric ($37{,}027$ distinct refined
classes: the $(5,4)$ run is contained in the $(5,5)$ run).  The
coarser, diagonal--blind refinement of Definition~\ref{def:refined}
was checked independently ($38{,}348$ verifications, $35{,}362$
distinct classes), with the same outcome.

We briefly indicate why area two lies beyond the method of this
paper as it stands.  With two active rows on diagonal $1$, the key
shielding phenomenon of Section~\ref{sec:caseB} fails: the earlier
of the two diagonal--one rows can attack the later one (a primary
attack, when its label is smaller), each can emit secondary attacks
to later diagonal--$0$ rows, depending on the labels, and in
particular a diagonal--one row
can now \emph{receive}, so their content variables can no longer be
carried by transparent markers; moreover the two rows may be
adjacent,
producing a second species of compound step whose operator involves
$\dd$ and $\ddb$ simultaneously.  The four--generator algebra
$\langle\Mu,\Mv,\dd,\ddb\rangle$ still hosts all the resulting
words, but the analogues of Lemma~\ref{lem:B} now form a
two--parameter family of coupled recursions which we have not closed
in general.  We expect the calculus to extend, and we propose
Conjecture~\ref{conj:refined}, rather than only
Conjecture~\ref{conj:exchange}, as the structurally correct
target.

\subsection{Verification}\label{ssec:verification}

The proofs above are self--contained and make no use of machine
computation; conversely, the area--two material of
Section~\ref{ssec:refined} is finite computational evidence for
Conjecture~\ref{conj:refined}, not a proved statement.
Independently of the proofs, the complete chain has been verified by
computer, in four layers, all independent of the proofs (the three
combinatorial layers are each compared against a common brute--force
enumeration of the objects $(a,w,S)$; the operator layer is a
stand--alone algebraic check):

\begin{enumerate}
\item \emph{Brute force.}  Direct enumeration of all objects,
attacks, valleys and $\dinv$; for every class $(m,k,\Xi,r)$ in
range, the polynomial $\Phi^{(r)}_{m,k,\Xi}$ was confirmed
symmetric, and the formula that $D$ is determined by the outside
scaffold data (Remark~\ref{rem:D}) was confirmed with key $(r,\Xi)$,
independent of $k$.
\item \emph{Dictionary.}  The receiver weights and validity rules
(V1)--(V5) of Lemma~\ref{lem:dict} were implemented as an
independent enumerator and reproduce every class--$k$ polynomial
exactly.
\item \emph{Operator identities.}  Every non--parameterized relation
of Lemma~\ref{lem:comm}, and the tested finite ranges of its
parameterized relations, were checked identity by identity, in exact
arithmetic and without truncation of the compared outputs, on every
monomial $u^{a}v^{b}z^{c}$ with $a+b\le4$ and $c\le2$.  The
parameter ranges are: $a\le6$ for the assertion
$\dd^a\onev=(-qzv;q)_a$ in Lemma~\ref{lem:comm}\textup{(3)},
$a,k\le2$ for the pull--through relation~\eqref{eq:Lpull} in
Lemma~\ref{lem:comm}\textup{(4)}, $g\le3$ for the iterated relation
$T^g\Mv=\Mv(T^{\langle q\rangle})^g$ in
Lemma~\ref{lem:comm}\textup{(6)}, and $k\le2$ in
Lemma~\ref{lem:comm}\textup{(7)}.  The same operator layer also
verified the expansion~\eqref{eq:qbinT} ($m\le6$), the symmetry of
$T^{m}\onev$ ($m\le 8$), Corollary~\ref{cor:TbarT} ($N\le6$), the
symmetry of $\BB(g,M)$ ($g,M\le 5$), and the recursion~\eqref{eq:Brec}
($g,M\le5$), all over $\mathbb{Z}[q,q^{-1}]$; the truncation bound of
the series engine exceeds the total degree of every compared series,
so no truncation occurs in this layer.
\item \emph{Assembly.}  The operator words of
Propositions~\ref{prop:flatA} and~\ref{prop:caseB} (including
every gate, glue chain, and $P$--placement of
Lemmas~\ref{lem:cluster}--\ref{lem:Pblocks}) were evaluated
symbolically and compared with the brute--force polynomial of every
nonempty \emph{refined} class arising in the enumeration, separately
(classes with no objects are $0$ by the conventions of
Section~\ref{sec:dictionary} and are not separately instantiated):
for the five pairs
\[
(n,M)\in\{(3,4),\,(4,4),\,(5,4),\,(5,5),\,(6,4)\}
\]
of path size and label--alphabet bound, in both multisets $0^n$ and
$0^{n-1}1$, a
total of $58{,}013$ nonempty refined--class verifications, with
exact agreement in every case; the $(5,4)$ classes recur within the
$(5,5)$ run, so the number of distinct classes is $53{,}054$.
\end{enumerate}

The area--two evidence of Section~\ref{ssec:refined} is a fifth,
independent computation, and a sixth checks
Remark~\ref{rem:perdiag} directly.  For a uniform area--one run it
groups the objects with $m=0^{n-1}1$, $n\le5$, labels $\le5$ by
$(r,k,\Xi,g,\eta)$, where $\eta$ is the gap of $P$ in Case~B and a
sentinel in Case~A.  This yields $21{,}957$ refined groups, all
symmetric; $7{,}048$ of them are genuine Case~B per--$P$--gap groups.
In particular, the $381$ objects in range whose glue chains anchor at
$P$ with length at least two (Lemma~\ref{lem:Pblocks}) lie in $167$ of
these Case~B groups, all symmetric.  The complete verification code is
provided as supplementary material: a consolidated runner for the
four layers above, together with stand--alone scripts for the
area--two and per--$P$--gap computations.

\subsection{Open problems}\label{ssec:open}

\begin{enumerate}
\item \emph{Area two.}  Prove Conjecture~\ref{conj:refined} (or
only Conjecture~\ref{conj:exchange}) for $m=0^{n-2}1^{2}$,
$n\ge3$.  As
explained in Section~\ref{ssec:refined}, the missing ingredient is a
closed family of symmetric series generalizing Lemma~\ref{lem:B} to
two interacting $\ddb$--insertions.
\item \emph{Schur positivity.}  Corollary~\ref{cor:slices} upgrades
the $t^{0}$ and $t^{1}$ slices from quasisymmetric to symmetric;
their monomial expansions are manifestly positive.  The valley Delta
conjecture predicts that these slices are Schur positive (indeed equal to the corresponding slices of
$\Delta'_{e_{n-k-1}}e_n$), and it would be very interesting to
establish Schur positivity directly, e.g.\ through an LLT--type
expansion of the refined classes.
\item \emph{A bijective witness.}  Proposition~\ref{prop:noR} rules
out a letter--local involution, not a global one.  Does the refined
exchange of Theorem~\ref{thm:main} admit a weight--preserving
involution on objects?  The recursion~\eqref{eq:Brec}, whose
correction term is a difference of manifestly symmetric pieces
weighted by the factor $1-q^{-1}$, not a nonnegative integral
combination, suggests that any such involution must
act non--locally.
\item \emph{Propagation.}  The implications of Iraci and Vanden
Wyngaerd \cite{IVWpush,IVWsquare} transport the full valley Delta
conjecture to its generalized and square versions.  Do the
slice symmetries of Corollary~\ref{cor:slices}, or the refined
exchange itself, propagate along the same schedules?  Relatedly,
D'Adderio and Iraci \cite{DAdderioIraci} derive consequences from a
symmetry hypothesis on (a refinement of) the \emph{extended} valley
side; extending Theorem~\ref{thm:main} to the partially labelled
setting would connect the two threads.
\end{enumerate}

\end{document}